\documentclass[10pt]{amsart}
\usepackage{graphicx,amssymb,amsfonts,amsmath,amsthm,newlfont}
\usepackage{epsfig,url}
\usepackage{color}

\usepackage[all,2cell]{xy} \UseAllTwocells \SilentMatrices

\newtheorem{thm}{Theorem}[section]
\newtheorem{cor}[thm]{Corollary}
\newtheorem{lem}[thm]{Lemma}
\newtheorem{prop}[thm]{Proposition}
\theoremstyle{definition}
\newtheorem{defn}[thm]{Definition}

\theoremstyle{remark}
\newtheorem{rem}[thm]{Remark}
\numberwithin{equation}{section}

\newcommand{\Z}{\mathbb Z}
\newcommand{\C}{\mathbb C}

\newcommand{\R}{\mathbb R}

\newcommand{\Pro}{\mathbb P}

\newcommand{\gr}{\mathrm{gr}}

\font \rus= wncyr10
\newcommand{\sha}{\, \hbox{\rus x} \,}

\newcommand{\Ho}{\mathcal{H}}

\newcommand{\MT}{\mathcal{MT}}
\newcommand{\M}{\mathcal{M}}

\newcommand{\Isom}{\mathrm{Isom}}

\newcommand{\zetam}{\zeta^{ \mathfrak{m}}}
\newcommand{\zetaa}{\zeta^{ \mathfrak{a}}}

\newcommand{\Q}{\mathbb Q}
\newcommand{\Li}{\mathrm{Li}}
\newcommand{\Lo}{\mathcal{L}}
\newcommand{\U}{\mathcal{U}}

\newcommand{\To}{\longrightarrow}

\newcommand{\A}{\mathbb{A}}
\newcommand{\G}{\mathbb{G}}

\newcommand{\tone}{\overset{\rightarrow}{1}\!}
\newcommand{\opi}{{}_0 \Pi_{1}}

\newcommand{\opo}{{}_0 \Pi_{0}}

\newcommand{\opz}{{}_0 \Pi_{z}}
\newcommand{\Aut}{Aut}

\newcommand{\dR}{\mathfrak{dr}}

\newcommand{\Or}{\mathcal{O}}

\newcommand{\UMT}{\mathcal{U}_{dR}}
\newcommand{\circb}{\, \underline{\circ}\, }

\newcommand{\Pemp}{ \!\Pe^{\mm, +}}

\newcommand{\mm}{\mathfrak{m} }
\newcommand{\am}{\mathfrak{a} }
\newcommand{\HH}{\mathbb{H} }

\newcommand{\Ao}{\mathcal{A} }

\newcommand{\id}{\mathrm{id} }

\newcommand{\ooi}{{}_0 1_1}
\newcommand{\ooz}{{}_0 1_z}

\newcommand{\pim}{ \pi_{\am,\mm+}}
\newcommand{\pidram}{ \pi_{\am, \dR}}

\newcommand{\LL}{\mathbb{L}}

\newcommand{\Lefschetz}{Lefschetz }

\newcommand{\LM}{{\LL}^{\mm}}
\newcommand{\LA}{{\LL}^{\am}}
\newcommand{\LDR}{{\LL}^{\dR}}

\newcommand{\Zo}{Z_o}

\newcommand{\per}{\mathrm{per}}

\newcommand{\Spec}{\mathrm{Spec} \,}
\newcommand{\PP}{\mathfrak{P}}

\newcommand{\Pe}{\mathcal{P}}

\newcommand{\sv}{\mathrm{sv}}
\newcommand{\svm}{\sv^{\mm}}

\newcommand{\dch}{\mathrm{dch}}

\newcommand{\Hom}{\mathrm{Hom}}

\newcommand{\hookdownarrow}{\mathrel{\rotatebox[origin=c]{-90}{$\hookrightarrow$}}}

\begin{document}
\date{3rd September 2013}
\author{Francis Brown}
\begin{title}[Single-valued motivic periods]{Single-valued periods and multiple zeta values}\end{title}
\maketitle
 \begin{abstract} The values at $1$ of single-valued multiple polylogarithms span a certain subalgebra of multiple zeta values.
 In this paper, the properties of this algebra are studied from the  point of view of motivic periods.
 \end{abstract}

 \section{Introduction}
 The goal of this paper is to study a special class of multiple zeta values which occur as the  values at $1$ of single-valued multiple polylogarithms. The latter were defined in 
\cite{BSVMP}  and generalize the Bloch-Wigner dilogarithm
 \begin{equation} \label{BW} 
 D(z) = \mathrm{Im}( \Li_2(z) + \log |z| \log (1-z))
 \end{equation} 
which is a single-valued version of $\Li_2(z)$,  to  the case of all multiple polylogarithms in one variable. These are defined for any integers $n_1,\ldots, n_r\geq 1$ by 
 $$\Li_{n_1,\ldots, n_r} (z)= \sum_{0 < k_1 < \ldots <  k_r} {z^{k_r}  \over n_1^{k_1}\ldots n_r^{k_r}}$$
 and are iterated integrals on $\Pro^1 \backslash \{0,1,\infty\}$ obtained by integrating along the straight line path from $0$ to $1$ along the real axis.  In the convergent case $n_r\geq 2$, their values at one are precisely 
Euler's  multiple zeta values 
\begin{equation} \label{introMZV} 
\zeta(n_1,\ldots, n_r ) = \sum_{0 < k_1 < \ldots < k_r} {1 \over n_1^{k_1}\ldots n_r^{k_r}}\ .
\end{equation} 
 The values at one of the single-valued multiple polylogarithms define an interesting sub-class of multiple zeta values, which we denote by 
 \begin{equation} \label{introSVMZV}Ê
 \zeta_{\sv}(n_1,\ldots, n_r )\in \R \ .
 \end{equation} 
They satisfy  $\zeta_{\sv}(2)=D(1)=0$, as one  immediately sees  from $(\ref{BW})$. 
 These numbers are, in a precise sense, the values of iterated integrals  on $\Pro^1 \backslash \{0,1,\infty\}$ which are obtained by integrating from $0$ to $1$ \emph{independently of all choices of path}.
   
    The numbers $(\ref{introSVMZV})$     have
  recently found several applications in physics, in:
   \begin{enumerate}
 \item O. Schnetz'  theory of graphical functions for  Feynman amplitudes \cite{SG}
  \item The coefficients of the closed super-string tree-level amplitude \cite{SS}
 \item Wrapping functions\footnote{I was informed  by D. Volin that \cite{Wrap} equation (63) is a single-valued MZV to order $g^{18}$.}  in $N=4$ super Yang-Mills \cite{Wrap}, 
 \end{enumerate}
 as well as in \cite{Duhr,  Dixon}, 
 and also in mathematics as the coefficients of Deligne's associator.  A general theme seems to be that a large class of (but not all)  Feynman integrals in 4-dimensional renormalisable quantum field theories lie in the subspace of single-valued multiple zeta values. This raises an interesting possibility of replacing general amplitudes with their single-valued versions (see \S3), 
 which should lead to considerable simplifications. 
 \vspace{0.1in}

  \subsection{Contents} In \cite{BMTZ}, motivic multiple zeta values $\zetam(n_1,\ldots, n_r)$ were defined as elements of a certain graded algebra $\Ho$, equipped with a period homomorphism
 $$\per: \Ho \To \C$$
 which maps $\zetam(n_1,\ldots, n_r)$ to $\zeta(n_1,\ldots, n_r)$. In this paper, 
 motivic versions of the single-valued numbers $(\ref{introSVMZV})$, denoted $\zetam_{\sv}(n_1,\ldots, n_r)$,  are defined.  They generate a subalgebra $\Ho^{\sv}\subset \Ho$.
  Its main properties can be summarized as follows.
 
 \begin{thm} There is a natural homomorphism 
 $\Ho \rightarrow \Ho^{\sv}$ which sends $\zetam(n_1,\ldots, n_r)$ to $\zetam_{\sv}(n_1,\ldots, n_r)$.  In particular, the $\zetam_{\sv}(n_1,\ldots,n_r)$ satisfy all motivic relations
 for multiple zeta values, together with the relation $\zetam_{\sv}(2)=0$.
 
 The algebra $\Ho^{\sv}$ is isomorphic to the polynomial algebra 
  generated by 
  $$\zetam_{\sv} (n_1,\ldots, n_r)$$
 where $n_i \in \{2,3\}$ and $(n_1,\ldots, n_r)$ is a Lyndon word (for the ordering $3<2$) of odd weight.
 Furthermore,  $\Ho^{\sv}$ is preserved under the action of the motivic Galois group.
 \end{thm}

 In particular, the numbers $\zeta_{\sv} (n_1,\ldots, n_r)$ satisfy the  same  double shuffle and associator relations as usual multiple zeta values,  and many more  relations besides: the space $\Ho^{\sv}$ is much smaller
 than $\Ho$ (\S\ref{sectExamples}).  By way of example:
 \begin{eqnarray}
 \zeta_{\sv}(2n+1)  &=  &2 \,\zeta(2n+1) \quad  \hbox{ for all } n\geq 1 \nonumber \\
 \zeta_{\sv}(5,3)  &=  & 14 \, \zeta(3) \zeta(5) \nonumber \\
  \zeta_{\sv}(3,5,3) & = &  2 \, \zeta(3,5,3) - 2 \, \zeta(3) \zeta(3,5) -10 \,   \zeta(3)^2 \zeta(5)\nonumber 
 \end{eqnarray}
 The reader who is only interested in the single-valued multiple zeta values  and not their motivic versions can turn directly to    \S\ref{sectAclassofMZVs} 
 for an elementary definition (which only uses the Ihara action \S\ref{sectIharaaction}), and \S\ref{sectExamples} for enumerative properties and examples.
 \\
 
\subsection{Motivic periods} Whilst writing this paper, it seemed a good opportunity to clarify certain concepts relating to motivic multiple zeta values. There are two conflicting notions of motivic multiple zeta values in the literature, 
one due to Goncharov \cite{GoMT} (for which the motivic version of $\zeta(2)$ vanishes), via the concept of framed objects in mixed Tate categories, and  another for which the motivic version of $\zeta(2)$ is non-zero   \cite{BMTZ}, later simplified by Deligne \cite{DLetter}. It can be paraphrased as follows: 
 
 \begin{defn} Let $\mathcal{M}$ be a Tannakian category of motives,  with two fiber functors $\omega_{dR}, \omega_{B}$.
 A \emph{motivic period} is an element of the affine  ring of   the torsor of periods 
$$\Pe^{\mm} = \Or(\Isom_{\mathcal{M}}(\omega_{dR}, \omega_B) )\ .$$
 \end{defn}
Given a motive $M \in \mathcal{M}$, and classes $\eta\in \omega_{dR}(M)$, $X\in \omega_B(M)^{\vee}$, the \emph{motivic period}   \cite{DLetter}  associated to this data is the function on 
$\Isom_{\mathcal{M}}(\omega_{dR}, \omega_B) )$ defined by 
$$ [M, \eta, X]^{\mm} := \phi \mapsto \langle \phi(\eta), X\rangle$$
This  definition is nothing other than the standard construction of the ring of  functions on the   Tannaka groupoid, and appears in a similar form in \cite{An}, \S23.5.  However, it is the interpretation and application of this concept  which is  of interest here; in particular the idea that one can  sometimes deduce results about periods from their motivic versions and vice-versa (see e.g.,  \cite{BMTZ}, \S4.1).
The ring of motivic periods is a bitorsor over the Tannaka groups $(G_{\omega_{dR}}, G_{\omega_{B}})$ and thus gives rise to  a Galois theory of motivic periods. 
In this paper, we only consider the special case where  $\mathcal{M} = \MT(\Z)$ is the category of mixed Tate motives over $\Z$: the generalization to other categories of mixed Tate motives \cite{DG} is relatively straightforward if one replaces $\omega_{dR}$ with the canonical fiber functor, and bears in mind that there can be several different Betti realizations.

In a similar vein, one can replace $\omega_{dR}, \omega_B$ with any pair of fiber functors, to obtain various different  notions of motivic period. 
 One can  consider the ring of
\emph{de Rham periods} $\Pe^{\dR}$, where we replace $\mm=(\omega_{dR}, \omega_{B})$ with $\dR=(\omega_{dR},\omega_{dR})$, and  a weaker notion of \emph{unipotent de Rham periods}  $\Pe^{\am}$ which are their restriction to the unipotent radical $U_{\omega_{dR}}$ of the Tannaka group
$G_{\omega_{dR}}$.  The latter are precisely the `framed objects' studied in \cite{BMS}, \cite{BGSV}, \cite{GoMT}.
Although there is no direct period (integration) map for de  Rham motivic periods, we construct a related notion  in \S3 in the case $\mathcal{M} = \MT(\Z)$,  which we call the single-valued motivic period.  It gives a well-defined homomorphism from unipotent de Rham periods to motivic periods 
$$\sv^{\mm}:\Pe^{\am} \To \Pe^{\mm}$$ 
Composing with the period map attaches a complex number to de Rham periods. This gives a transcendental pairing between  a de Rham cohomology class and a de Rham homology class.
 In the case of $\MT(\Z)$, the numbers one obtains are precisely the single-valued multiple zeta values
$(\ref{introSVMZV})$.
Since the  definition  of $\sv^{\mm}$ requires nothing more than   complex conjugation and the weight grading,   it  comes  perhaps  as a surprise that this map is already so intricate in this special case (see, for example  $(\ref{zmsvexamples})$). 
 \\

 Section 2 consists of generalities on motivic and de Rham periods, and section 3 defines the motivic single-valued map $\sv^{\mm}$. 
 The remainder of the paper applies this construction to the case of the motivic fundamental group of $\Pro^1 \backslash \{0,1,\infty\}$. Section 4 consists of reminders, and \S5 aims to give a completely
 elementary definition of $\zeta_{\sv}$. Section 6 defines the motivic versions $\zetam_{\sv}$ and  \S\ref{SVMPrevisited}  reconstructs the single-valued multiple polylogarithms from the point of view of the unipotent fundamental group.
 Section 7 applies the main theorem of \cite{BMTZ} to deduce structural results about $\Ho^{\sv}$. 
 \\
 
 \emph{Acknowledgements}. Many thanks to P. Deligne and O. Schnetz for asking me essentially the same question: what are the coefficients of Deligne's associator, and what are the values at one of the
 single-valued multiple polylogarithms?
 This work was partially supported by ERC grant 257638, and written whilst a visiting scientist at Humboldt University.
 Many thanks also to C. Dupont for discussions, and  especially to O. Schnetz for  extensive numerical computations \cite{SI}.

 \subsubsection{Conventions}
 All tensor products are over $\Q$ unless expressly stated otherwise. 
 
 \section{Generalities on periods and mixed Tate motives} \label{sectreminders}

See \cite{DG}, \S2 for the  background material on mixed Tate motives required in this section.
Much of what follows applies  to any category of mixed Tate motives over a number field, provided that one replaces the de Rham fiber functor with the canonical fiber
functor $\omega_n(M) = \Hom_{\MT} ( \Q(-n), \gr^{W}_{2n} M)$ (see \cite{DG}, \S1.1).

\subsection{Mixed Tate motives over $\Z$} 
Let $\M=\MT(\Z)$ denote the Tannakian category of mixed Tate motives over $\Z$ \cite{DG}. 
Its   canonical fiber functor   is equal to the fiber functor $\omega_{dR}$ given by the de Rham realization, and it is equipped with  a fiber functor $\omega_B$ given by the Betti realization with respect to the 
unique embedding $\Q \hookrightarrow \C$. Let $G_{dR}$ and $G_B$ denote the corresponding Tannaka groups. 
They are affine group  schemes over $\Q$. We shall mainly focus on $G_{dR}$.  

The action of $G_{dR}$ on $\Q(-1)\in \M$ defines a map
$G_{dR} \rightarrow \G_m$
whose kernel is denoted by $\UMT$. It is a pro-unipotent affine group scheme over $\Q$.  Furthermore,
since $\omega_{dR}$ is graded, $G_{dR}$ admits  a 
 decomposition as a semi-direct product (\cite{DG}, \S2.1)
\begin{equation} \label{Gdrdecomp}
 G_{dR} \cong  \UMT \rtimes \mathbb{G}_m \ .
 \end{equation}Ê
A mixed Tate motive  $M\in \M$  can be represented by  a finite-dimensional graded $\Q$-vector space $M_{dR} = \omega_{dR}(M)$ equipped with an action of 
$\UMT$ which is compatible with the grading. We shall write $(M_{dR})_n$ for the component in degree $n$, i.e., 
$(M_{dR})_n = (W_{2n} \cap F^n) M_{dR}$. The  Betti realization  of $M$, denoted $M_B = \omega_B(M)$, is a finite-dimensional  $\Q$-vector space  equipped with an increasing filtration $W_\bullet M_{B}$. 

The two are related by  a canonical comparison isomorphism
\begin{equation} \label{BdR} 
 \mathrm{comp}_{B, dR} : M_{dR} \otimes_{\Q} \C \overset{\sim}{\To} M_{B}\otimes_{\Q} \C
 \end{equation} 
which can be computed by integrating differential forms. We shall often  use the fact that $\Q(0) \in \M$ has rational periods, i.e., that 
\begin{equation} \label{Q0ratperiods}
\mathrm{comp}_{B,dR}: \Q(0)_{dR} \overset{\sim}{\To} \Q(0)_{B}\ .
\end{equation}

In general, given any pair of fiber functors $\omega_1,\omega_2$ on $\M$, let 
$$\PP_{\omega_1,\omega_2} = \Isom( \omega_{1}, \omega_{2})$$
denote the set of isomorphisms of fiber functors from $\omega_{1}$ to $\omega_2$. It is a scheme over $\Q$, 
and is a bitorsor over $(G_{\omega_1}, G_{\omega_2})$, where $G_{\omega_i} = \PP_{\omega_i,\omega_i}$ is the Tannaka group scheme relative to $\omega_i$, for $i=1,2$.
  The comparison map defines a complex point
$$\mathrm{comp}_{B,dR} \in \PP_{\omega_{dR}, \omega_B}(\C) \ .$$

\subsection{Motivic periods}
There are two conflicting  notions of motivic multiple zeta values in the literature, one due to \cite{GoMT} and the other due to \cite{BMTZ}. 
One can reconcile the two definitions  with minimal damage to  existing terminology as follows.

\begin{defn} \label{defnmotperiods} Let $\omega_1, \omega_2$ be two fiber functors on $\M$.  Let $ M \in \mathrm{Ind}\, (\M)$ and let  $\eta \in \omega_1(M)$, and $X \in \omega_2(M)^{\vee}$. A \emph{motivic    period} of  $M$ of type $(\omega_1,\omega_2)$
\begin{equation}\label{motper} 
[M, \eta, X]^{\omega_1,\omega_2} \in \Or(\PP_{\omega_1, \omega_2} )
\end{equation}
is  the   function $\PP_{\omega_1,\omega_2} \rightarrow \Q$ defined by
$  \phi \mapsto  \langle \phi(\eta), X \rangle = \langle \eta, {}^t \phi(X) \rangle .$
\end{defn}

One can clearly extend definition \ref{defnmotperiods} to other Tannakian categories (of motives), but only  $\M=\MT(\Z)$ will be considered in this paper. 

Definition  \ref{defnmotperiods} in the case $(\omega_1, \omega_2) = (\omega_{dR}, \omega_B)$  is due to Deligne \cite{DLetter}, and simplifies the definition in \cite{BMTZ}. Since  this is the case of primary interest for us, we shall call  $(\ref{motper})$
a \emph{motivic period}, and denote the pair $(\omega_{dR}, \omega_B)$ simply by $\mm$.
We shall also consider the case   $(\omega_1, \omega_2) = (\omega_{dR}, \omega_{dR})$. We shall call the corresponding  period $(\ref{motper})$ a \emph{de Rham period}, and denote the pair
$(\omega_{dR}, \omega_{dR})$  by $\dR$.

\begin{defn} Let $M \in \M$ and let $\omega_1,\omega_2$ be a pair of fiber functors as above.
We shall denote the space of all  motivic periods of type $(\omega_1,\omega_2)$ by
\begin{equation}\label{Pedef}
 \Pe^{\omega_1,\omega_2} = \Or(\PP_{\omega_1,\omega_2})\ ,
\end{equation}
 and  we shall write 
$\Pe^{\omega_1,\omega_2}(M) $  for the $\Q$-subspace of  $\Pe^{\omega_1, \omega_2} $  
 spanned by the  motivic periods of $M$ of type $(\omega_1,\omega_2)$.
\end{defn}
It follows from $(\ref{Pedef})$ that the set of all motivic periods forms an algebra over $\Q$.
The set of fiber functors on $\M$ form a groupoid with respect to composition
$$\PP_{\omega_1,\omega_2} \times \PP_{\omega_2,\omega_3} \rightarrow \PP_{\omega_1,\omega_3}$$
for any three fiber functors $\omega_1,\omega_2,\omega_3$.  
Dualizing, we obtain a  coalgebroid structure on  spaces of motivic periods:
\begin{eqnarray} \label{Hopfalgebroid} 
\Pe^{\omega_1,\omega_3}  \To \Pe^{\omega_1,\omega_2}  \otimes \Pe^{\omega_2,\omega_3} \ ,
\end{eqnarray} 
which, in the case $\omega_1 = \omega_2= \omega_{dR}$,  and $\omega_3= \omega_B$ becomes a coaction:
\begin{equation} \label{Pemdrcoaction} 
\Delta_{\dR,\mm}: \Pe^{\mm}  \To  \Pe^{\dR}  \otimes \Pe^{\mm}\ .
\end{equation}
By the definition of the Tannaka group, 
\begin{equation} \label{GdRasSpecP}
G_{dR} = \Spec(\Pe^{\dR})\ ,
\end{equation} 
and so $(\ref{Pemdrcoaction} )$ defines in particular an action of $G_{dR}(\Q)$ on the space of motivic periods $\Pe^{\mm}$.
 Since $\omega_{dR}$ is graded, the (left) action of $\G_m \subset G_{dR}$ corresponds to a grading on $\Pe^{\omega_{dR}, \omega}$ for any  fiber functor $\omega$.  
We shall sometimes call the degree the weight, in keeping with standard terminology for multiple zeta values. It is one half of  the Hodge-theoretic weight, i.e., $\Q(-n)$ has weight $n$.

The notions of motivic and de Rham periods are fundamentally different.  In the case of motivic periods, pairing with the element $(\ref{BdR})$ defines the period homomorphism 
\begin{equation}\label{perhom} 
\per : \Pe^{\mm} \To \C\ .
\end{equation} 
 The point $1 \in G_{dR}$ defines a map $\Pe^{\dR} \rightarrow \Q$. The period map \emph{per se}  is not available for de Rham periods, although  we shall define a substitute in  \S\ref{sectSVMP}.

\subsection{Formulae}   By the main construction of Tannaka theory  (\cite{DeTannaka}, \S4.7),
motivic periods  are spanned by   symbols $[M, \eta, X]^{\omega_1, \omega_2}$, where $M \in \M$, $\eta  \in \omega_1(M)$, and $X\in \omega_2(M)^{\vee}$, modulo  the equivalence relation generated by 
\begin{equation}\label{Framedequiv}
[M_1, \eta_1, X_1]^{\omega_1,\omega_2}  \sim [M_2, \eta_2, X_2]^{\omega_1,\omega_2} 
\end{equation}
for every morphism $\rho: M_1 \rightarrow M_2$ such that $\eta_2 = \omega_1(\rho) \eta_1$, and 
$X_1 = \omega_2(\rho)^{t} X_2$.

The  multiplication on motivic periods is given concretely by the formula:
\begin{eqnarray} \label{permult} 
 \Pe^{\omega_1,\omega_2}(M_1) \times  \Pe^{\omega_1,\omega_2}(M_2)  & \To &   \Pe^{\omega_1,\omega_2}(M_1\otimes M_2) \\
{[}M_1,\eta_1,X_1] ^{\omega_1, \omega_2}Ê\times [M_2,\eta_2,X_2]^{\omega_1, \omega_2}   &= & [M_1\otimes M_2,\eta_1\otimes \eta_2,X_1\otimes X_2]^{\omega_1, \omega_2} \ . \nonumber 
 \end{eqnarray}
In particular, if   $M$ is an algebra object  in $\mathrm{Ind}\, (\M)$, then $\Pe^{\omega_1,\omega_2}(M)$ is a commutative ring, and its spectrum is an affine scheme over $\Q$.

Given three fiber functors $\omega_1, \omega_2, \omega_3$, the Hopf algebroid structure $(\ref{Hopfalgebroid})$ can be computed explicitly by the usual coproduct formula for endomorphisms 
\begin{eqnarray} \label{mmcoaction}
\Delta_{\omega_1,\omega_2; \omega_2,\omega_3}: \Pe^{\omega_1,\omega_3}(M)  &\To  &\Pe^{\omega_1,\omega_2}(M) \otimes \Pe^{\omega_2,\omega_3}(M) \\
{[}M,\eta, X]^{\omega_1, \omega_3} & \mapsto &  \sum_{v} [M, \eta, v^{\vee}]^{\omega_1,\omega_2} \otimes  [M,  v,X]^{\omega_2,\omega_3} \nonumber
\end{eqnarray}
where $\{v\}$ is  a basis of $ \omega_2(M)$ and $\{v^{\vee}\}$ is the dual basis. The previous formula  does not depend on the choice of basis.  In the case $\omega_1 = \omega_2 = \omega_{dR}$,  equation
$(\ref{mmcoaction})$   gives the following formula for the weight of a  motivic period:
\begin{equation} \label{degreeofM}
 \deg\, [M, \eta, X ]^{\omega_{dR}, \omega} = m\ ,
 \end{equation}
whenever  $\eta\in  (M_{dR})_m$ has degree $m$, and $\omega_3=\omega$ is any fiber functor.

Finally, given a motivic period $[M, \eta, X]^{\mm} \in \Pe^{\mm}$,  its period is given by 
\begin{equation}
\per([M, \eta, X]^{\mm})= \langle \mathrm{comp}_{B,dR}(\eta), X \rangle\in \C\ .
\end{equation}ÊIn principle  it can always be computed by integrating a differential form representing $\eta$ along
a topological cycle representing $X$.

\subsection{Unipotent de Rham periods} There is yet another  notion of de Rham period which is obtained by restricting to the unipotent radical $U_{dR} \subset G_{dR}$.

\begin{defn} Let $v\in M_{dR}$, and $f \in M_{dR}^{\vee}$. A \emph{unipotent de Rham period} is the image of 
$[M,v,f]^{\dR}$ under the map $\Or(G_{dR}) \rightarrow \Or(U_{dR})$.  Denote it by
$$[M, v, f]^{\am} \in \Or(U_{dR})$$
and  denote the ring of unipotent de Rham periods by 
\begin{equation} \label{Pea}
\Pe^{\am}\cong \Or(U_{dR})\ .
\end{equation}
\end{defn}

Unipotent de Rham periods are  equivalent  to the notion of framed objects in mixed Tate categories considered, for example, in (\cite{GoMT}, \S2). 
There is a natural map 
$$\pidram : \Pe^{\dR} \To \Pe^{\am}\ ,$$
and hence, by taking     $\omega_1=\omega_2 = \omega_{dR}$ and $\omega_3= \omega_B$  and  restricting the left-hand factor of the right-hand side of $(\ref{mmcoaction})$  to $\Pe^{\am}$, we obtain  a  coaction 
\begin{eqnarray} \label{Deltacoaction}
\Delta_{\mm,\am}: \Pe^{\mm}  & \To & \Pe^{\am} \otimes \Pe^{\mm} 
\end{eqnarray}
The action of $\G_m$ by conjugation gives  $\Or(U_{dR})$ a grading. A (non-zero) unipotent de Rham period $Ê[ M, v_m, f_n]^{\am}$
is homogeneous of degree
\begin{equation}\label{degmminusn}
 \deg {Ê[ M, v_m, f_n]^{\am} =m-n }
\end{equation}
whenever $v_m \in (M_{dR})_{m}$, and $f_n \in ((M_{dR})_n)^{\vee}$. 
 Note that the formula only agrees with  $(\ref{degreeofM})$ when $n=0$. 
Since  $\Or(U_{dR})$ has  weights $\geq 0$,  $Ê[ M, v_m, f_n]^{\am} $ vanishes if $m<n$.
With these definitions,  the coaction  $(\ref{Deltacoaction})$
is  homogeneous in the weight.

\begin{rem}  
In \cite{GoMT}, it is assumed that  one   framing, namely $f_n$,   is in degree  zero.  This defines a smaller space of de Rham periods for a given motive $M$, and the corresponding coproduct 
formula  requires an extra Tate twist  in the left-hand factor. 
\end{rem}

\subsection{Motives generated by motivic periods} It is very    useful to think of a space of motivic periods  $ \Pe^{\mm}(M)$ 
as a motive in its own right.

\begin{defn}
Let $\xi \in \Pe^{\mm}$ be  a motivic period. Let $M({\xi})_{dR}$ denote the graded $\Or(U_{dR})$-comodule it generates  via the coaction  $(\ref{Deltacoaction})$.

By the Tannakian formalism,  this is the de Rham realization of a  motive  we denote by $M({\xi})\in \M$. Define  \emph{the motive generated by the motivic period $\xi$} to be $M(\xi)$.
\end{defn}

\begin{lem} \label{lemxiperiodofMxi} For any $\xi \in \Pe^{\mm}$, $\xi$ is a motivic period of $M(\xi)$.
\end{lem} 
\begin{proof}  If we represent $\xi$ by a triple $[M , \eta, X]^{\mm}$,  then the de Rham orbit $G_{dR}\eta$ defines a submotive 
$M^1 \subset M$ such that $M^1_{dR} = G_{dR} \eta$. We have an equivalence
$$ [M^1, \eta, X^1]^{\mm} \overset{\sim}{\To} [M, \eta, X]^{\mm}= \xi   $$
where $X^1$ is the image of $X$ in $(M^1_B)^{\vee}$. Now define $M^2$ to be the quotient motive of $M^1$
whose de Rham realization is   $M^1_{dR} / (\PP_{dR,B} X^1)^{\perp}$. Then
$$[M^1, \eta, X^1]^{\mm} \overset{\sim}{\To} [M^2, \eta_2, X^1]^{\mm} $$
are equivalent, 
where $\eta_2$ is the image of $\eta$ in $M^2_{dR}$. 
In particular, $\xi$ is a motivic period of $M^2$.  The de Rham realization of $M_2$ is exactly
$${ G_{dR} \eta \over  G_{dR} \eta \cap (\PP_{dR, B} X)^{\perp}}$$ 
which  is isomorphic to  the $G_{dR}$-module generated by the function
$[M,  \eta, X]^{\mm}\in \Pe^{\mm}$. Therefore  $M^2_{dR} = M(\xi)_{dR}$ and hence $M^2= M(\xi) $.   
\end{proof}

Thus $M(\xi)$ is the smallest subquotient motive  $M'$ of $M$  such that 
$\xi \in \Pe^{m}(M')$.

\subsection{Geometric periods} \label{sectgeomcase} The notions of de Rham and motivic periods can be  related to each other via the  following algebra of geometric periods.

\begin{defn}\label{defngeomper} Let $\Pemp \subset \Pe^{\mm}$ be the largest graded subalgebra of $\Pe^{\mm}$ such that:

i). $\Pemp$ has weights $\geq0$,

ii). $\Pemp$  is a comodule under  $\Pe^{\am}$, i.e., 
\begin{equation} \label{Pempcoaction}
\Delta_{\am,\mm}: \Pemp \To \Pe^{\am} \otimes \Pemp\ .
 \end{equation}
\end{defn}
Suppose that $M\in \M$ has positive weights, i.e., $W_{-1} M =0$.   
 Then
$$\Pe(M) \subset \Pemp\ . $$

\begin{lem} \label{lempemgen} The algebra $\Pemp$ is generated by  the motivic periods of $M$, where $M$ has non-negative weights ($W_{-1}M=0$).
\end{lem}
\begin{proof}  The graded vector space $\Pemp$ is an $\Or(U_{dR})$-comodule by $(\ref{Pempcoaction})$. It is therefore the de Rham realization of an object $\mathbb{P} \in \mathrm{Ind}(\M)$
which has weights $\geq 0$ by $\ref{defngeomper}$, $i).$
By  lemma \ref{lemxiperiodofMxi}, every $\xi \in \Pemp$ is a motivic period of $\mathbb{P}$.  
\end{proof} 

It follows from   lemma  $\ref{lempemgen}$ that 
\begin{equation} \label{Pemp0isQ}
\Pemp_0 \cong \Q\ .
\end{equation}
This is because a motivic period of weight zero of a motive $M$ satisfying $W_{-1}M =0$, is equivalent to 
a period of $\Q(0)$, which is rational. Note that the isomorphism $(\ref{Pemp0isQ})$ uses  $\mathrm{comp}_{B,dR}$ via $(\ref{Q0ratperiods})$.
 As a consequence,  there is an augmentation map
$$ \varepsilon : \Pemp\To  \Q$$
given by projection onto $\Pemp_0$. This defines a map
\begin{equation} \label{PemtoPea}
\pim: \Pemp \To \Pe^{\am} 
\end{equation} 
 by composing the coaction $\Delta_{\am,\mm}: \Pemp \To \Pe^{\am} \otimes \Pemp$ with $\varepsilon$.  The map $\pim$  respects the weight gradings, and  is an isomorphism in weight zero  $\pim: \Pemp_0 \cong  \Pe^{\am}_0.$
\\

The map $\pim$ can be computed another way. 
Let $M$ satisfy $W_{-1}M =0$. 
Then $W_0 M$ is a direct sum of copies of  $\Q(0)$, which has rational periods $(\ref{Q0ratperiods})$. We have
$$ \gr^{W}_0 M_{dR} = W_0 M_{dR}\quad  \overset{ \mathrm{comp}_{B,dR} } {\To} \quad  W_0 M_B \hookrightarrow M_B \ . $$ 
Since $M_{dR}$ is graded, we can  first apply the projection  $M_{dR} \rightarrow \gr^{W}_0 M_{dR}$ and then apply the previous  map. This defines a rational comparison morphism
 \begin{equation}\label{c0def}   c_{0}: M_{dR} \To M_{B} \ .\end{equation}
 Then $(\ref{PemtoPea})$ is given by the formula
 \begin{eqnarray} \label{PemtoPeaexplicit}
\pim: \Pe^{\mm}(M)  &\To  &  \Pe^{\am}(M)  \\
{[}M, \eta, X]^{\mm}  &\mapsto &  [M, \eta, {}^t c_{0}(X)]^{\am} \nonumber
\end{eqnarray}
for all $ \eta \in M_{dR},  X \in M_B^{\vee}.$  It only depends on the restriction of $X$  to $W_0 M_B$.

\subsection{Example: the \Lefschetz motive.} \label{sectLefschetz}
Let $M= H^1(\G_m) \cong \Q(-1) $. Then $M_{dR}= H_{dR}^1(\G_m; \Q) \cong  \Q\, \omega_0$,  and $M_B^{\vee} =H_1(\C^{\times}; \Q) = \Q \gamma_0 $, 
where  $\omega_0= {dz \over z}$, and $\gamma_0$ is a loop winding around $0$ in the positive direction. Denote the \Lefschetz motivic period by 
\begin{equation} \label{LLdef}
  \LM=  [M, [\omega_0], [\gamma_0] ]^{\mm} 
  \end{equation}
whose weight is one and whose  period is $$\per(\LM)=\int_{\gamma_0} \omega_0 = 2i \pi. $$  The element $\LM$ is invertible in $\Pe^{\mm}$. We use the notation $\LM$ in order to avoid the rather ugly alternative $(2 i \pi)^{\mm}$.
Define the \Lefschetz de Rham period by 
\begin{equation} \label{LLdefdR}
  \LDR=  [M, [\omega_0], [\omega_0]^{\vee} ]^{\dR} \ .
  \end{equation}
It is  group-like for  the coproduct on $\Or(G_{dR})$: $\Delta_{\dR,\dR} \LDR =\LDR \otimes \LDR$.  Since
$U_{dR}$ acts trivially on $\Q(-1)$, the unipotent de Rham  \Lefschetz  period is trivial:
\begin{equation} \label{LLdefAM}
  \LA= \pidram (\LDR)=1 \ .
  \end{equation}
By definition,   $\LDR$ can be viewed as a coordinate on $\G_m$, and 
\begin{equation} \label{Gmasspec}
\G_m \cong \Spec \Q [ (\LDR)^{-1}, \LDR]\ .
\end{equation}
 On the other hand,  $\gr^W_0 M =0$, so  $c_0([\gamma_0])=0$ and therefore
\begin{equation} \label{pimlmzero} 
 \pim(\LM)=  [M, [\omega_0], c_0([\gamma_0]) ]^{\am} =0\ .
 \end{equation} 
By $(\ref{mmcoaction})$, the coaction $\Delta_{\dR, \mm}: \Pe^{\mm} \rightarrow \Pe^{\dR} \otimes \Pe^{\mm}$
acts on the motivic \Lefschetz period by $\Delta_{\dR, \mm}\LM= \LDR \otimes \LM$. By $(\ref{LLdefAM})$ the coaction $\Delta_{\am,\mm}:\Pe^{\mm} \rightarrow \Pe^{\am} \otimes \Pe^{\mm}$ satisfies
\begin{equation} \label{LMcoaction}
\Delta_{\am,\mm}(\LM) = 1 \otimes \LM\ .
\end{equation}

\subsection{Structure of de Rham periods}
The fact that $G_{dR}$ is  a semi-direct product  $(\ref{Gdrdecomp})$  implies that $G_{dR} \cong U_{dR} \times \G_m$ as  schemes, and hence
\begin{equation} \label{PdRdecomp}
\Pe^{\dR} \cong \Pe^{\am} \otimes \Q[  (\LDR)^{-1} , \LDR]\ .
\end{equation}
The coaction of $\Pe^{\am}$ on the right-hand side is given by the formula $\Delta_{\am, \dR} (\LDR) = 1 \otimes \LDR$, by $(\ref{LLdefAM})$.
 Equivalently, the map $\G_m \rightarrow G_{dR}$ induces a projection
 \begin{equation} \label{pilldr}
\pi_{\LL,\dR}: \Pe^{\dR} \To \Q[  (\LDR)^{-1} , \LDR]\ ,
\end{equation}
 or explicitly  
$ \pi_{\LL,\dR} \big( [M, v, f]^{\dR} \big)  =   f(v) \, (\LDR)^{n}$ if $ v\in  (M_{dR})_n$.
 The isomorphism  $(\ref{PdRdecomp})$ is then induced by  
composing the coaction $\Delta_{\am,\dR}: \Pe^{\dR} \rightarrow \Pe^{\am} \otimes \Pe^{\dR}$ with $\id \otimes \pi_{\LL,\dR}$.

\begin{rem} \label{remPa0} If $v  \in   (M_{dR})_m$ and $f \in (( M_{dR})_n)^{\vee}$ are of degrees $m>n$ respectively,   then the image of $[M, v, f]^{\am}$ under the implied 
section $\Pe^{\am} \rightarrow \Pe^{\dR}$ is $[M(n), v, f]$, where $v$ now sits in degree $m-n$, and $f$ in degree zero. 
The literature on framed mixed Tate objects  essentially  identifies $\Pe^{\am}$ with its image $\Pe^{\am,0}$  in $\Pe^{\dR}$. 
\end{rem}

\subsection{Structure of motivic periods}
Rather than using the canonical isomorphism of fiber functors $\mathrm{comp}_{dR,B}$, which is defined over $\C$, we prefer
to choose  rational isomorphisms, which are non-canonical.
\begin{prop} \label{propexistsisom} There exists an isomorphism of fiber functors  from $\omega_B$ to $\omega_{dR}$.
\end{prop}
\begin{proof} See the proof of proposition 8.10 in \cite{DeP1}. 
\end{proof}
By choosing such an element $s\in \mathrm{Isom}(\omega_{B}, \omega_{dR})$, we obtain an isomorphism
\begin{equation} \label{sisomonbeta}
 \mathrm{Isom}(\omega_{dR}, \omega_{B}) \overset{\sim}{\To} \mathrm{Isom}(\omega_{dR}, \omega_{dR}) \ .
 \end{equation} 
Dually, this gives  $s^t:\Pe^{\dR} \overset{\sim}{\To} \Pe^{\mm}$, and so $(\ref{PdRdecomp})$ gives a non-canonical isomorphism
$$s^t:   \Pe^{\am} \otimes \Q[ (\LDR)^{-1}, \LDR]   \overset{\sim}{\To} \Pe^{\mm}\ .$$
By \S\ref{sectLefschetz} we can  assume that $s^t( \LDR) = \LM$, and write the previous isomorphism as
\begin{equation} \label{Pemdecomp}
\qquad \qquad  \qquad \quad \quad \Pe^{\mm} \cong   \Pe^{\am} \otimes \Q[ (\LM)^{-1}, \LM]   \qquad \qquad (\hbox{depending on } s)\ .
\end{equation}
It is compatible with the  coaction $\Delta_{\am,\mm}: \Pe^{\mm} \rightarrow \Pe^{\am} \otimes \Pe^{\mm}$, and the weight gradings.

\begin{cor}  There is a non-canonical decomposition
\begin{equation} \label{PempDecomp}
  \Pemp \cong \Pe^{\am} \otimes \Q[  \LM] \ .
 \end{equation} 
\end{cor}
\begin{proof}
The decomposition $(\ref{Pemdecomp})$ is  induced by $(\id \otimes \pi_{\LM}) \circ \Delta_{\am,\mm}$,
 where  $\pi_{\LM}$ is given by $(s^t)^{-1}$ followed by $(\ref{pilldr})$, and $\LDR\mapsto \LM$.  Since $\Pemp$ has weights $\geq 0$, and 
$\LM$ has weight $1$, the restriction  of   $(\ref{Pemdecomp})$ to $\Pemp$ gives an injective map 
$$  \Pemp \To \Pe^{\am} \otimes \Q[\LM]\ .$$
The image of $\Pe^{\am} \otimes \Q[\LM]$ in $\Pe^{\mm}$ has weights $\geq 0$, and is $G_{dR}$-stable.  Since $\Pemp$ is the largest subalgebra of $\Pe^{\mm}$
with this property, the previous map is an isomorphism.
\end{proof}
Sending  $\LM$ to zero in  $(\ref{PempDecomp})$ gives back the map $\pim:  \Pemp \rightarrow \Pe^{\am}$.

\subsection{Real Frobenius} \label{sectRealstruct} Since there is a unique embedding from $\Q$ to $\C$,
complex conjugation defines the real Frobenius
$ c : M_B \rightarrow M_B$. 
 It induces an involution 
 \begin{eqnarray}
c: \Pe^{\mm} & \To & \Pe^{\mm} \\
{[}M, \eta, X]^{\mm}  &\mapsto &  [M, \eta, c(X)]^{\mm} \nonumber 
\end{eqnarray}
which is compatible, via the period homomorphism,  with  complex conjugation on $\C$.
If $\Delta_{\am,\mm}: \Pe^{\mm}\rightarrow \Pe^{\am} \otimes \Pe^{\mm}$ denotes the coaction, then clearly $\Delta_{\am,\mm} c = (\id \otimes c) \Delta_{\am,\mm}$.  
Since
 $$c (\LM ) =  -\LM $$
it follows that $c$ acts on a decomposition $(\ref{PempDecomp})$ by multiplying $(\LM)^n$ by $(-1)^n$.

 \begin{cor} \label{corRealFrob} If   \, $\Pemp_{\R}$  (respectively\,  $\Pemp_{i \R}$) denotes the subspace of $\Pemp$ of  invariants (anti-invariants) of the map   $c$, 
 then we have 
 \begin{eqnarray}
 \Pemp_{i \R}  & \cong  & \Pemp_{\R}  \,  \LM \  \nonumber  \\
\hbox{ and }  \qquad  \Pemp_{\R}  &\cong & \Pe^{\am} \otimes \Q[ (\LM)^2]  \nonumber 
  \end{eqnarray} 
 with respect to some choice of  decomposition $(\ref{PempDecomp})$.
 \end{cor}

\subsection{Universal comparison map} 
The identity map $\id: \Pe^{\mm} \rightarrow \Pe^{\mm}$ defines a canonical element 
in  $ (\Spec \Pe^{\mm})(\Pe^{\mm})$ which we denote by 
$$\mathrm{comp}^{\mm}_{B,dR} \in  \Isom_{\omega_{dR}, \omega_{B}}(\Pe^{\mm})\ .$$ 
It reduces to the usual comparison map $\mathrm{comp}_{B,dR}$ after applying
the period homomorphism to the coefficient ring $\Pe^{\mm}$. 
As a formula, it is given for $M \in \M$ by 
\begin{eqnarray} \label{universalBdRforM}
\mathrm{comp}_{B,dR}^{\mm} : M_{dR} &   \To &  M_B \otimes \Pe^{\mm} (M) \\
\eta & \mapsto & \sum_{x} x \otimes [M, \eta, x^{\vee}]^{\mm} \nonumber 
\end{eqnarray}
where the sum ranges over a basis $\{x\}$ of $M_B$, and $\{x^{\vee}\}$ is the dual basis. 
We can also write $(\ref{universalBdR})$ as  an isomorphism after tensoring with all motivic periods:
\begin{equation}  \label{universalBdR}
\mathrm{comp}_{B,dR}^{\mm} : M_{dR} \otimes \Pe^{\mm} \overset{\sim}{\To} M_{B} \otimes \Pe^{\mm}  \ .
\end{equation}
In the other direction, we have a universal map
$$\mathrm{comp}_{dR,B}^{\mm} : M_{B}    \To   M_{dR} \otimes \Pe^{\omega_B,\omega_{dR}} (M) $$
which is defined in a similar way. It will not be used here.

The universal comparison maps can be used to compare the action of the de Rham motivic Galois group $G_{dR}$  with the action of the 
Betti Galois group $G_B$ on $\Pe^{\mm}$.

\section{Single-valued motivic periods} \label{sectSVMP}
The single-valued period is an analogue of the period homomorphism for de Rham periods.
First, we construct a well-defined map (the single-valued motivic period)
$$\svm : \Pe^{\am} \To \Pemp$$
and define the single-valued period to be 
$$\sv : \Pe^{\am} \To \C $$
 by composing with the usual period $\per:\Pe^{\mm}\rightarrow \C$. The map $\sv$ is similar to what is sometimes referred to as the `real period' in the literature \cite{GoHyp}, \S4. Since multiple zeta values are already
 real numbers, this terminology could lead to confusion, so we prefer not to use it.
Note that the single-valued periods of a motive $M$ are not in fact periods of $M$, but elements of the algebra generated by the  periods of  $M$.

In the latter half  of the paper, we shall compute the single-valued versions of motivic multiple zeta values using the motivic fundamental group of $\Pro^1 \backslash \{0,1,\infty\}$.  
More precisely, we compute the following map $$\Pemp \overset{\pim}{\To} \Pe^{\am}  \overset{\sv^{\mm}}{\To} \Pemp$$
on the subspace $\Ho \subset \Pemp$ of motivic multiple zeta values.

\subsection{ Single-valued motivic periods}
The weight-grading on $\Pe^{\mm}$ is given by  an action of $\G_m$, which we shall denote by $\tau$. Thus  $\tau(\lambda)$ is  the map which in weight $n$ 
acts via multiplication by $\lambda^n$, for any $\lambda \in \Q^{\times}$.

\begin{defn} Let $\sigma : \Pe^{\mm} \To \Pe^{\mm}$ be the involution 
\begin{equation} \label{defsigma}
 \sigma=\tau(-1) \, c 
 \end{equation} 
 where $c$ is the real Frobenius of \S\ref{sectRealstruct}. For example,  $\sigma(\LM)=\LM$.
\end{defn}
 
 \begin{rem} \label{remsigmabar} If $\Delta_{\am,\mm} : \Pe^{\mm} \rightarrow \Pe^{\am} \otimes \Pe^{\mm}$ denotes the coaction, then 
 $$\Delta_{\am,\mm} \sigma =( \overline{\sigma}Ê\otimes \sigma) \circ \Delta_{\am,\mm}$$ 
where $\overline{\sigma}: \Pe^{\am} \rightarrow \Pe^{\am}$ is given by the action of $\tau(-1)$  on $\Pe^{\am}$ by conjugation. In other words, $\overline{\sigma}$  acts by multiplication by  $(-1)^n$ in degree 
$n$, where the degree is $(\ref{degmminusn})$.
\end{rem}

 Consider the following affine scheme over $\Q$:
 $$\Pro= \Spec(\Pe^{\mm})\qquad \qquad \big( = \Isom(\omega_{dR}, \omega_B) \big)\ .$$ 
 The coaction $\Pe^{\mm} \rightarrow \Pe^{\dR} \otimes \Pe^{\mm}$ defines an action we denote by $\circ$:
$$\circ: G_{dR} \times \Pro \To \Pro$$
and makes  $\Pro$ a torsor over $G_{dR}$ (by proposition \ref{propexistsisom}).  The maps $\id, \sigma : \Pe^{\mm} \rightarrow \Pe^{\mm}$
can be viewed as  elements
$\id, \sigma \in \Pro(\Pe^{\mm})$.

\begin{defn} Define $\sv^{\mm}$  to be the unique element of $G_{dR}(\Pe^{\mm})$ such that
\begin{equation} \label{svmdef1}
\sv^{\mm} \circ \sigma = \id\ . \end{equation}
\end{defn}

Let us compute $\sv^{\mm}(\LDR)$. Recall that   $\Delta_{\dR, \mm} \LM = \LDR \otimes \LM$. Thus 
$$ \LM = (\sv^{\mm} \circ \sigma)( \LM) = \mu (\sv^{\mm} \otimes \sigma) ( \LDR \otimes \LM) =  \svm(\LDR) \sigma(\LM)$$
where $\mu$ denotes multiplication.  Since $\sigma (\LM)= \LM$, we deduce that
 $\sv^{\mm}(\LDR) =1$.  Therefore $\sv^{\mm}$ actually lies in the image of $U_{dR}(\Pe^{\mm})$  in  $G_{dR}(\Pe^{\mm})$ and we can view it as a homomorphism
$\sv^{\mm} : \Pe^{\am} \To \Pe^{\mm}.$ Even more precisely, we have:

\begin{prop}  For all $g \in G_{dR}$, and $\xi \in \Pe^{\am}$, 
\begin{equation} \label{svconjugation}
\sv^{\mm} (c_g  \xi ) = g\,  \sv^{\mm} ( \xi)
\end{equation}  
where $c_g$ denotes the action of $g\in G_{dR}$ on $\Pe^{\am}$ by twisted conjugation 
$$c_g(\xi) = g\, \xi\,  \overline{g}^{-1}$$
where $\overline{g} = \tau(-1) \,g\,  \tau(-1)$. 
In particular,
$\sv^{\mm}$ defines a  homomorphism
\begin{equation}
\sv^{\mm} : \Pe^{\am} \To \Pemp\ 
\end{equation}Ê
which is homogeneous for the weight-gradings on both sides.
\end{prop}

\begin{proof}  
For any $g\in G_{dR}$, 
define  $\sv^{\mm}_{\!g} : \Pe^{\am} \rightarrow \Pe^{\mm}$  to be $\sv^{\mm}_{\!g}(x) = g \,\sv^{\mm}(x)$,   and similarly define  $\sigma_g, \id_g\in G_{dR}(\Pe^{\mm})$, where 
$\sigma_g(x) =g\, \sigma(x)$, $\id_g(x) = g \, \id(x)$, and
the action of $g$ is on the ring of  coefficients $\Pe^{\mm}$.
  By  the definition  $(\ref{svmdef1})$  of $\sv^{\mm}$,   we have
$$\sv^{\mm}_{\!g} \circ \sigma_g = \id_g\ .$$
Clearly $\id_g = g$, but     $\sigma_g = \overline{g} \circ \sigma$ by remark \ref{remsigmabar}. Therefore 
$$\sv^{\mm}_{\!g} \circ  \overline{g} \circ \sigma = g\ .$$
Since $\Pro $  is a torsor over $G_{dR}$, this has the unique solution $\sv^{\mm}_{\!g} = g \circ \sv^{\mm} \circ \overline{g}^{\circ -1}$, which is precisely $(\ref{svconjugation}).$ Since the weight-grading on $\Pe^{\am}$ is given by conjugation  by $g$ for $g\in  \G_m$
and $\overline{g}=g$ for such $g$ (because $\G_m$ is commutative),  we
deduce from $(\ref{svconjugation})$ that $\sv^{\mm}$ is homogeneous in the weight. In particular,  since $\Pe^{\am}$ has weight $\geq 0$, the image of $\sv^{\mm}$ has weight $\geq 0$, is stable under $G_{dR}$, and hence is contained in $\Pemp$.
\end{proof}

The formula $(\ref{svconjugation})$ can be translated into coactions as follows.  Let 
$$ \Lo^{\am} = {\Pe_{>0}^{\am} \over \Pe_{>0}^{\am}\Pe_{>0}^{\am}}$$
denote the Lie coalgebra of indecomposable elements of $\Pe^{\am}$. Projecting from $\Pe^{\am}_{>0}$ to  $\Lo^{\am}$ defines  infinitesimal  versions of the usual coactions $(\ref{mmcoaction})$
\begin{eqnarray}
\delta : \Pe^{\mm}  &\To & \Lo^{\am} \otimes \Pe^{\mm} \nonumber \\
\delta_L : \Pe^{\am}  &\To & \Lo^{\am} \otimes \Pe^{\am} \nonumber \\
\delta_R : \Pe^{\am}  &\To &  \Pe^{\am} \otimes \Lo^{\am} \cong \Lo^{\am} \otimes \Pe^{\am} \nonumber 
\end{eqnarray}
where $\delta_L, \delta_R$ are obtained from the left and right coactions of $\Pe^{\am}$ on itself. Then
\begin{equation}Ê\label{deltasv}
\delta \, \sv^{\mm} ( \xi)  = \sv^{\mm} (\delta_L \xi)  +  (\overline{S} \otimes \id)\,  \sv^{\mm}( \delta_R\xi )\ , \end{equation}
where $\overline{S}:  \Lo^{\am}\rightarrow  \Lo^{\am} $ is multiplication by $(-1)^n$ in degree $n$ followed by the 
infinitesimal version of the  antipode $S: \Lo^{\am} \rightarrow  \Lo^{\am}$.  

\begin{defn} Let $\Pe^{\sv}\subset \Pemp$ denote the image of the map $\sv^{\mm}$.  We shall call it the ring of single-valued motivic periods.
\end{defn} 

\subsection{Properties of  the single-valued motivic period}
For  computations,  it is convenient to trivialize the torsor $\Pro$ as follows.
By proposition \ref{propexistsisom}, we can choose an isomorphism of fiber functors 
$ s' \in \Isom(\omega_B, \omega_{dR}).$ It defines an isomorphism $(\ref{sisomonbeta})$
\begin{equation} \label{sPdRtoPmm}Ê
 s:\Or(G_{dR}) = \Pe^{\dR} \overset{\sim}{\To} \Pe^{\mm} \ ,
 \end{equation} 
  where $s= (s')^t$,  which we  view as a $\Pe^{\mm}$-valued point of $G_{dR}$, denoted  $s \in G_{dR}(\Pe^{\mm})$.
The action of the involution $(\ref{defsigma})$ on its coefficients will be denoted by $^{\sigma}$. 

Then $\sv^{\mm} \in G_{dR}(\Pe^{\mm})$ can be computed via the expression  
\begin{equation} \label{svmdef}
\sv^{\mm} = s  \circ ({}^{\sigma}\! s)^{\circ -1} 
\end{equation} 
where the inversion and multiplication $\circ$ take place in the group $G_{dR}$.  

\begin{rem}
To check that $(\ref{svmdef})$ is well-defined, 
let $s_1', s_2'\in \Isom(\omega_{B}, \omega_{dR})(\Q)$. Since the latter is a 
$(G_B, G_{dR})(\Q)$-bitorsor, there exists an element $\rho' \in G_{dR}(\Q)$ such that $s_2'=\rho' s_1'$.
Transposing gives  $s_2 = s_1\circ  \rho$, where $\rho$ is the image of $\rho'$ in $ G_{dR} (\Pe^{\mm})$ via $\Q \subset \Pe^{\mm}$. In particular, ${}^{\sigma}\!\rho=\rho$ since its coefficients
are rational of weight zero.  Thus
$$  s_2  \circ ({}^{\sigma}\! s_2)^{\circ -1}  = 
s_1 \circ \rho   \circ ({}^{\sigma}\! \rho)^{\circ -1} \circ ({}^{\sigma}\! s_1)^{\circ -1}=s_1  \circ ({}^{\sigma}\! s_1)^{\circ -1}\ ,$$
and $(\ref{svmdef})$ is well-defined, as expected. 
\end{rem}

\begin{defn} Let $\Pe^{\mm,0} \subset \Pemp $ denote the subring of motivic periods
$$\Pe^{\mm,0} = \bigcap_{s} s(\Pe^{\am})$$
where $s$ ranges over  maps $s:\Pe^{\am} \rightarrow \Pemp$ induced by decompositions $(\ref{PempDecomp})$.
Since $\pim$ is injective on the image of such an $s$, it is injective on $\Pe^{\mm,0}$.

\end{defn}

\begin{lem}  \label{imageofsvm} We have $Ê\Pe^{\sv} \subset \Pe^{\mm,0}$. 
 In particular, 
$\pim: \Pe^{\sv}\rightarrow \Pe^{\am} $ is injective.  

The  compositum  $\pim Ê\sv^{\mm} : \Pe^{\am} \rightarrow \Pe^{\am}$ is given by the element  
 \begin{equation} \label{idcircsigma}
 \id \circ \sigma^{\circ -1}\in U_{dR}(\Pe^{\am})\ . 
 \end{equation}

\end{lem}
\begin{proof}  A choice of isomorphism $(\ref{Pemdecomp})$ defines  a map $s: \Pe^{\am} \rightarrow \Pemp$ (and hence an element $s\in U_{dR}(\Pe^{\mm})$) which we can use  to compute $\sv^{\mm}$. 
 By a similar argument to the discussion preceding  $(\ref{svmdef})$,  except that we work in $U_{dR}$ instead of $G_{dR}$,  we have  $\sv^{\mm} = s  \circ ({}^{\sigma}\! s)^{\circ -1}$. The  coefficients of $s$, and a fortiori
 $\sv^{\mm}$, lie in the subspace $s(\Pe^{\am}) \subset \Pemp$.  This proves the first statement. 
 
 For the second statement, observe that  $\pim s $ is the identity map on $\Pe^{\am}$, and therefore $\pim \sv^{\mm} =   \id   \circ ({}^{\sigma}\id)^{\circ -1}$, which gives exactly $(\ref{idcircsigma}).$
\end{proof}

In particular,  the map $\pim  \sv^{\mm}: \Pe^{\am} \rightarrow \Pe^{\am}$ is not the identity, and    has a large kernel. The previous lemma will be used in \S \ref{sectStructure} to determine the structure of $\Pe^{\sv}$.

\subsection{Formulae} We can translate $(\ref{svmdef})$ into a  formula  as follows.  Let $H$ be any connected, commutative graded Hopf algebra, with coproduct $\Delta:H \rightarrow H \otimes H$.  Denote its reduced coproduct
by $\Delta^{(1)} = \Delta - 1 \otimes \id - \id \otimes 1$, and its iterated coproduct by 
$$\Delta^{(n)}= (\id \otimes \Delta^{(n-1)}) \Delta^{(1)} =  (\Delta^{(n-1)} \otimes \id) \Delta^{(1)} $$
for $n\geq 2$. 
By a version of Sweedler's notation we can write
$$\Delta^{(n-1)} (x) = \sum_{(x)} x^{(1)} \otimes \ldots \otimes x^{(n)}$$
where all elements $x^{(i)}$ have degree $\geq 1$. In any such Hopf algebra, the antipode can be written
$S=\sum_{n\geq 1} (-1)^n \mu_n \Delta^{(n)}$, where $\mu_n: H^{\otimes n} \rightarrow H$ is the $n$-fold multiplication, and 
$\Delta^{(0)}$ is the identity. In Sweedler's notation, this is 
$$S(x) = \sum_{n\geq 1}  (-1)^n \sum_{(x)}   x^{(1)} \ldots x^{(n)}$$

In order to apply the above, we must consider  the Hopf algebra $\Pe^{\am,0}$ defined to be the image of $\Pe^{\am}$ in  
$\Pe^{\dR}$  via the isomorphism $(\ref{PdRdecomp})$.  It is a commutative graded Hopf algebra, spanned by objects $[M,v,f]^{\dR}$ where $f$ is in degree $0$ by remark \ref{remPa0}.   A choice of homomorphism   $(\ref{sPdRtoPmm})$ gives a homomorphism
 $s: \Pe^{\am,0} \rightarrow \Pe^{\mm}$. Applying the previous remarks  to  the Hopf algebra $\Pe^{\am,0}$, we obtain
 \begin{eqnarray}
 \svm: \Pe^{\am,0} & \To & \Pe^{\mm} \nonumber \\
\svm(\xi)  & =  &  s(\xi) + \widetilde{s}(\xi) + \sum_{(\xi)} s(\xi^{(1)})  \widetilde{s}(\xi^{(2)})  \nonumber
 \end{eqnarray}
 where we use the notation
 $$ \widetilde{s}(\xi) =  \sum_{n\geq 1} \sum_{(\xi)} (-1)^n\,  {}^{\sigma}\!s(\xi^{(1)}) \ldots   {}^{\sigma}\!s(\xi^{(n)})\ .$$
 This essentially follows from the formula $(\ref{svmdef})$, after replacing  $G_{dR}$ with  $U_{dR}$.

Concretely, the motivic periods $s(\xi)$ can be computed as follows.  If $M\in \M$,  then the element $s$ defines an isomorphism 
$s: M_{B} \rightarrow M_{dR}.$ Let $s^{t}$ be its transpose. Then  if $v\in M_{dR}$  and $f\in ((M_{dR})_0)^{\vee}$ is  of degree $0$,
$$ s [M, v, f]^{\am} = [M, v, s^{t} f]^{\mm}\ . $$
In general we must first compose with  $\Pe^{\am} \overset{\sim}{\rightarrow}\Pe^{\am,0}$, which introduces a Tate twist by remark 
 \ref{remPa0}.
 It follows that the single-valued  motivic periods of $M$ are in  fact products of motivic periods of  Tate twists of $M$.


\section{The motivic fundamental group of $\Pro^1 \backslash \{0,1,\infty\}$} \label{sect21} 

The main references for   this section are \cite{DeP1}, \cite{DG}, \cite{BMTZ}. A  good  introduction  can be found in \cite{DB}. 

Let $X=\Pro^1\backslash \{0,1,\infty\}$, and let 
$\tone_0, -\tone_1$
denote the  tangential base points on $X$  given by the  vector  $1$ at $0$, and  the  vector $-1$ at $1$.  Denote the  motivic fundamental torsor of paths on $X$ by
$$\opi^{\mm}= \pi_1^{\mm}( X, \tone_0, -\tone_1)\ .$$ 
It is an affine scheme in the category $\MT(\Z)$. This means that there is a finitely generated  commutative algebra object
$\Or(\opi^{\mm}) \in \mathrm{Ind} \, \MT(\Z)$, and $\omega(\opi^{\mm})$ is defined to be $\Spec\!$ of the commutative algebra  $\omega(\Or(\opi^{\mm}))$, for any fiber functor $\omega$.

 We shall denote the de Rham 
realization   $\omega_{dR}(\opi^{\mm})$ of $\opi^{\mm}$ simply by 
$$\opi=  \Spec \Or(\opi)  \ ,$$
where $\Or(\opi)$ is isomorphic to $H^0( B( \Omega_{\log}^{\bullet}(\Pro^1,\{0,1,\infty\};\Q)))$, where $B$ is the bar complex.  Writing $e^0$ for ${dz\over z}$ and $e^1$ for ${dz\over 1-z}$, we can identify the latter with the 
graded $\Q$-algebra
$$\Or(\opi)\cong \Q \langle e^0, e^1 \rangle\ .$$
Its underlying vector space is spanned by the set of words $w$ in the letters $e^0, e^1$, together with the empty word, and  the multiplication is given by
 the shuffle product $\sha:  \Q \langle e^0, e^1 \rangle\otimes  \Q \langle e^0, e^1 \rangle \rightarrow  \Q \langle e^0, e^1 \rangle $ which is defined recursively by
 $$ (e_i w ) \sha (e_j w') = e_i (w \sha e_j w') + e_j (e_i w \sha w')$$
 for all  words $w, w'$ in $\{e_0,e_1\}$ and $i,j\in \{0,1\}$. The empty word will be denoted by $1$. It is the unit for the shuffle product: $1\sha w = w\sha 1$ for all $w$.

  The de Rham realization  $\opi$ is therefore isomorphic to $\Spec \Q\langle e^0, e^1\rangle$. It is  
 the affine scheme over $\Q$ which to any commutative unitary $\Q$-algebra $R$ associates the set of group-like formal power series in two non-commuting variables $e_0$ and $e_1$
$$\opi(R) = \{ S \in R\langle \!\langle e_0, e_1 \rangle \! \rangle^\times:  \Delta S= S \widehat{\otimes} S\}\  . $$
Here, $\Delta$ is the completed coproduct  $R\langle \!\langle e_0, e_1 \rangle \! \rangle \rightarrow R\langle \!\langle e_0, e_1 \rangle \! \rangle \widehat{\otimes}_R R\langle \!\langle e_0, e_1 \rangle \! \rangle$ for which the elements $e_0$ and $e_1$ are primitive: $\Delta e_i = 1\otimes e_i + e_i \otimes 1$ for $i=0,1$. 
 
Since the bar complex is augmented, we have an augmentation map $\opi \rightarrow \Q$ which is the projection onto the empty word. Dually, this corresponds to an element denoted
$$\ooi \in \opi(\Q)$$
which is called the canonical de Rham path from $\tone_0$ to $-\tone_1$.

On the other hand, the Betti realization of $\opi^{\mm}$ is the affine  scheme over $\Q$ given by the Mal\^cev  completion of the topological fundamental torsor of paths
$$\omega_B( \opi^{\mm}) \cong \pi^{un}_1(X(\C),\tone_0, -\tone_1) \ .$$
There is a natural map $\pi_1(X(\C),\tone_0, -\tone_1)  \rightarrow  \pi^{un}_1(X(\C),\tone_0, -\tone_1) (\Q)$.

\subsection{Drinfeld's associator} There is a canonical  straight line path  (`droit chemin')
\begin{equation}\label{dchdef}
\dch \in  \pi_1(X(\C),\tone_0, -\tone_1)
\end{equation}
 which therefore corresponds to an element   in $\omega_B(\opi^{\mm})$.  Via the isomorphism $(\ref{BdR})$, it defines  an element 
in $\opi(\C)$, which we denote by
$$ Z(e_0,e_1)  \in   \opi(\C)$$
It is  precisely Drinfeld's associator, and is given in low degrees by the formula
\begin{equation} \label{DrinfeldAssoc}
 Z(e_0,e_1)   =  1 + \zeta(2) {[}e_1,e_0] +  \zeta(3) ([e_1,[e_1,e_0]]+ [e_0,[e_0,e_1]] ) + \ldots   
 \end{equation}
In general, the coefficients are multiple zeta values. In fact, $(\ref{DrinfeldAssoc})$ is  the non-commutative generating series of (shuffle-regularized) multiple zeta values
$$Z(e_0,e_1) =  \sum_{w \in \{e_0,e_1\}^\times}  \zeta(w) \, w \ . $$
The coefficient $\zeta(w)$ is given by the regularized iterated integral
$$\zeta(e_{a_1}\ldots e_{a_n} ) = \int_{\dch}  \omega_{a_1} \ldots \omega_{a_n} \qquad \hbox{ for } a_i \in \{0,1\} $$
where  $\omega_0 = {dt \over t}$ and $\omega_1 = {dt \over 1-t}$, and the integration begins on  the left. One shows that the $\zeta(w)$ are linear combinations of multiple zeta values 
$(\ref{introMZV})$ and that for $n_r \geq 2$, 
$$\zeta(e_1 e_0^{n_1-1} e_1  e_0^{n_2-1} \ldots   e_1e_0^{n_r-1}) = \zeta(n_1,\ldots, n_r)\ .$$

\subsection{The Ihara action}      \label{sectIharaaction}
    Since  $\Or(\opi)$ is the de Rham realization of an  Ind-object in the category $\MT(\Z)$, it inherits  an action of the motivic Galois group
   $$  \UMT \times \opi \To \opi\ . $$
   The action of $\UMT$ on the  element $\ooi\in \opi$  defines a map
$$ \label{gtog1} g\mapsto g(\ooi): \UMT \To \opi\ ,
  $$  and one shows  \cite{DG},  \S5.8, that  the action of  $\UMT$ on $\opi$ factors through a map   
\begin{equation} \label{circonopi}
       \circ : \opi \times \opi  \To  \opi  
\end{equation} 
 which, on the level of formal power series, is given by the following formula
 \begin{eqnarray} \label{Iharaact}
 R\langle \langle e_0, e_1 \rangle \rangle^{\times}  \times R\langle \langle e_0, e_1 \rangle \rangle    &  \To  & R\langle \langle e_0, e_1 \rangle \rangle     \\
F(e_0,e_1) \circ G(e_0, e_1) & = &  G(e_0,  F(e_0, e_1) e_1 F(e_0, e_1)^{-1}) F(e_0, e_1) \nonumber 
\end{eqnarray}
 which was first considered by Y. Ihara. The action $(\ref{circonopi})$ makes $\opi$ into a torsor over $\opi$ for $\circ$. More prosaically, given two invertible  formal power series
 $G,H$, one can  solve $F\circ G = H$ for $F$  recursively by writing equation $(\ref{Iharaact})$ as
 \begin{equation}\label{Frecursive}
 F = G(e_0, Fe_1F^{-1})^{-1} H\end{equation}
  If all the  coefficients in $F$  of words of length $\leq N$ have been  determined, 
  then the coefficients of $F e_1 F^{-1}$, and hence $G(e_0, Fe_1 F^{-1})$ are determined up to length $N+1$.   Equation $(\ref{Frecursive})$ determines the coefficients of $F$ in length $N+1$. A similar
  recurrence based on the number of occurrences of $e_1$ in a  word (the depth) sometimes allows one to write down closed formulae in low depth and in all weights (\S\ref{sectExamples}).

\subsection{Motivic multiple zeta values}

Let   $\dch_B\in \omega_B(\opi^{\mm})(\Q)  $ denote   the Betti image of the straight line path $(\ref{dchdef})$. It defines an element
$\dch_B \in \omega_B(\opi^{\mm})^{\vee}$.
Let $w$ be any word in  $\{e^0,e^1\}$. It defines an element $w \in \Or(\opi) \cong \Q \langle e^0, e^1\rangle$, the de Rham realization of $\Or(\opi^{\mm})$.

\begin{defn} The \emph{motivic multiple zeta value} $\zetam(w)$ is the motivic period 
$$\zetam(w) = [ \Or(\opi^{\mm}), w,  \dch_B]^{\mm} \ .$$
The \emph{algebra of motivic multiple zeta values} $\Ho\subset  \Pe^{\mm}$ is  the graded $\Q$-algebra spanned by the $\zetam(w)$,
 i.e., the image of the map $w\mapsto \zetam(w): \Q\langle e^0, e^1 \rangle \rightarrow \Pe^{\mm}$. 
\end{defn}
Since $\Or(\opi^{\mm})$ has weights $\geq 0$, and is stable under $U_{dR}$, it follows that $\Ho \subset \Pemp$ by definition \ref{defngeomper}.
Thus $\Ho = \bigoplus_{n\geq 0} \Ho_n$ is positively graded, and there is a natural map
\begin{eqnarray} \label{canonicalmaptoH} 
\Q \langle e^0, e^1\rangle  &\To & \Ho \\
w  & \mapsto &  \zetam(w) \ .\nonumber 
\end{eqnarray}
which is a homomorphism for the shuffle product.  The period map  $(\ref{perhom})$  yields
\begin{eqnarray}
\per : \Ho & \To & \R \\
\zetam(w) & \mapsto & \zeta(w) \nonumber 
\end{eqnarray} 
and the periods of motivic multiple zeta values are the usual multiple zeta values.

There is a corresponding notion of unipotent de Rham multiple zeta value. Instead of  $\dch_B$, we now take  a de Rham framing
$\ooi \in \opi(\Q) \subset  \Or(\opi)^{\vee}$. 

\begin{defn} The \emph{unipotent de Rham multiple zeta value} $\zetaa(w)$ is 
$$\zetaa(w) = [ \Or(\opi^{\mm}), w,  \ooi]^{\am} \ .$$
The \emph{algebra of  unipotent de Rham multiple zeta values} $\Ao\subset \Pe^{\am}$ is  the graded $\Q$-algebra spanned by the $\zetaa(w)$,
 i.e., the image of the map $w\mapsto \zetaa(w): \Q\langle e^0, e^1 \rangle \rightarrow \Pe^{\am}$. 
\end{defn}
Since $\Or(\opi^{\mm})$ has non-negative weights,  and because the de Rham image $Z(e_0,e_1)$ of $\dch$   has leading term $1$, we verify  that 
$${}^t c_0 ( \dch_B) = \ooi \ .$$
By   equation $(\ref{PemtoPea})$, we deduce a surjective homomorphism
\begin{eqnarray} \label{HtoA}
\pim:\Ho  & \To \!\!\!\!\!\! \!\! \To &  \Ao  \\
\zetam(w) & \mapsto & \zetaa(w) \nonumber 
 \end{eqnarray} 
The motivic multiple zeta values $\zetam(w)$ were defined in \cite{BMTZ}, and simplified by Deligne \cite{DLetter}.  The unipotent de Rham multiple zeta values $\zetaa(w)$ are equivalent to the `motivic multiple
zeta values' considered  in \cite{GoMT}. 

\begin{rem}  It is important to note that  $\zetam(2) \neq 0$,  whereas  $\zetaa(2)=0$ (\cite{BMTZ}).
\end{rem}

The algebra $\Ao=\bigoplus_{n\geq 0} \Ao_n$ is again positively graded, and is a commutative Hopf algebra by $(\ref{Hopfalgebroid})$.  
We have a commutative diagram

$$
\begin{array}{ccc}
\Ho  &   \To &  \Ao \otimes \Ho   \\
\downarrow  &   &  \downarrow   \\
 \Ao  & \To   & \Ao \otimes \Ao  
\end{array}
$$
Let us write $\A= \Spec(\Ao)$, and $\HH=\Spec(\Ho)$. Then $\A$ is a pro-unipotent  affine group scheme over $\Q$, which embeds in $\HH$ via $(\ref{HtoA})$, and  acts upon it on the left:
\begin{equation} \label{AHtoH}
\A \times \HH \To \HH\ .
\end{equation}

\subsection{Compatibility with the Ihara action}
The fact that the action of the motivic Galois group factors through the Ihara action (\S\ref{sectIharaaction}) can be expressed by the following commutative diagram, where the maps $\A \hookrightarrow \HH \hookrightarrow  \opi$ are induced by $(\ref{HtoA})$, $(\ref{canonicalmaptoH})$:
$$
\begin{array}{ccc}
 \A\,\,\times \, \,\HH  &  \To   & \HH   \\
 \hookdownarrow\, \,\, \quad \,  \,\,\,\hookdownarrow &    & \hookdownarrow    \\
 \opi \times \opi & \To  &  \opi  
\end{array}
$$
and the map along the bottom is the Ihara action $\circ:   \opi \times \opi \rightarrow   \opi$.  
Dually, we have the following commutative diagram \cite{BMTZ}: 
$$
\begin{array}{ccc}
 \Or(\opi)  & \To  &   \Or(\opi)   \otimes \Or(\opi) \\
\downarrow   &    & \downarrow    \\
 \Ho   &  \To   & \Ao \otimes \Ho   
\end{array}
$$
where the map along the top is the  Ihara coaction, which can be effectively replaced  \cite{BrDepth}, with an explicit formula which is due to Goncharov (who proved it for the unipotent de Rham 
periods $\zetaa$, i.e. modulo $\zeta(2)$'s, but in fact gives the correct coaction for $\zetam$ also. See \cite{BrDepth} for  a  direct and very short proof using the Ihara coaction).

\subsection{The motivic Drinfel'd associator} In this section, all $\Hom$'s are in the category of commutative unitary $\Q$-algebras.

\begin{defn} Define the \emph{motivic version of the Drinfel'd associator} by
$$Z^{\mm}(e_0, e_1) = \sum_{w \in \{ e_0, e_1\}^{\times}} \zetam(w) w \qquad \in \qquad    \opi(\Ho) \subset \Ho \langle \langle e_0, e_1 \rangle \rangle$$
Define the \emph{unipotent de Rham version of the Drinfel'd associator} by
$$Z^{\am}(e_0, e_1) = \sum_{w \in \{ e_0, e_1\}^{\times}} \zetaa(w) w \qquad \in \qquad \opi(\Ao) \subset  \Ao \langle \langle e_0, e_1 \rangle \rangle \ .$$
\end{defn}

It is useful to view $Z^{\mm}$ as a  morphism via the following general nonsense.
 For any commutative unitary ring $R$, we have an isomorphism
 $$ \Hom( \Q\langle e_0, e_1\rangle , R) \overset{\sim}{\To} R\langle \langle e_0, e_1 \rangle \rangle\  .$$
Via this isomorphism, we see that $Z^{\mm}$ is simply the image of the canonical  map $(\ref{canonicalmaptoH})$.
Composing with the  canonical map $(\ref{canonicalmaptoH})$ gives a map:
$$\Hom(\Ho, R) \To \Hom(\Q\langle e_0, e_1 \rangle, R)\To  R \langle\langle  e_0, e_1 \rangle \rangle$$
which is simply another way to write $\HH(R) \hookrightarrow \opi(R)$. Setting $R = \Ho$, we can view  the motivic Drinfel'd associator
as the image of the identity map
\begin{equation} \label{idasZ} 
\id_{\Ho} \in \Hom(\Ho, \Ho) \qquad \longrightarrow  \qquad  Z^{\mm} \in \Ho \langle\langle  e_0, e_1 \rangle \rangle \ . 
\end{equation} 
The usual Drinfel'd associator is the image of the element
$$\per \in \Hom(\Ho, \C) \qquad \longrightarrow  \qquad  Z \in \C \langle\langle  e_0, e_1 \rangle \rangle \ .  $$
The  unipotent de Rham  Drinfel'd associator is the image of  the map $(\ref{HtoA})$:
$$ \pim \in \Hom(\Ho, \Ao) \qquad \longrightarrow  \qquad  Z^{\am} \in \Ao \langle\langle  e_0, e_1 \rangle \rangle \ .  $$

\subsection{Decomposition with respect to $\zetam(2)$'s} \label{sectdecompzeta(2)}

\begin{lem} We have
$$\zetam(2) = - {  (\LM)^2 \over 24}$$
\end{lem} 
\begin{proof} The action of $U_{dR}$ on $\zetam(2)$ is trivial, by \cite{BMTZ}, \S3.2. Since $\zetam(2)$ has weight 2, it is equal to a rational multiple
of $(\LM)^2$. The rational multiple is determined by applying the period map and using Euler's formula $\zeta(2) = \pi^2/6$.  
\end{proof} 
The analogue of Euler's theorem is false for de Rham periods, since  $\zeta^{\dR}(2)=0$.

\begin{lem}  (\cite{BMTZ}, \S2.3) There is a non-canonical isomorphism 
\begin{equation} \label{HoisAotenszeta2}
\Ho \cong \Ao \otimes \Q[\zetam(2)]
\end{equation}
\end{lem}
\begin{proof} Since the motive $\Or(\opi^{\mm})$ has weights $\geq 0$,  and is stable under $U_{dR}$, $\Ho$ is contained in $\Pemp$.  
Futhermore,  since  the path $\dch$ is   invariant under  complex conjugation
we deduce that $\Ho \subset \Pemp_{\R}$. By  corollary \ref{corRealFrob},  there is an injective map
\begin{equation} \label{Hodecomp}
\Ho  \To \Pe^{\am} \otimes \Q[(\LM)^2] \cong \Pe^{\am} \otimes \Q[\zetam(2)]  \ .
\end{equation}
which is compatible with the $\Pe^{\am}$-coaction.  
Since, by definition,  $\pim(\Ho) = \Ao$, and because  $\zetam(2) \in \Ho$,  the image of  $(\ref{Hodecomp})$
is equal to   $\Ao \otimes \Q[\zetam(2)]$.  
\end{proof}
A choice of  decomposition  $(\ref{HoisAotenszeta2})$ defines  a homomorphism  $\Zo^{\mm} : \Ao \rightarrow \Ho$, and via the augmentation on $\Pe^{\am}$,  a homomorphism $\gamma^{\mm}: \Ho \rightarrow \Q[\zetam(2)]$.

\begin{cor} There exist  elements  $\gamma^{\mm} \in \HH(\Q[\zetam(2)])$, and $\Zo^{\mm}\in \A(\Ho)$ such that 
 \begin{equation} \label{ZisZocircgamma} Z^{\mm}= \Zo^{\mm} \circ \gamma^{\mm} \ .
 \end{equation} 
 \end{cor} 
\begin{proof}The map $\Ho \rightarrow \Ao \otimes \Q[\zetam(2)] \rightarrow \Ho$ is the identity, where the second map is 
$\mu ( \Zo^{\mm} \otimes \id )$ and $\mu$ denotes multiplication. This implies that 
$\id_{\Ho} = \mu ( \Zo^{\mm} \otimes \gamma^{\mm}) \Delta$, 
i.e.,  $\id_{\Ho}$ is the convolution product of $\Zo^{\mm}$ and $\gamma^{\mm}$. This is exactly
 $(\ref{ZisZocircgamma})$, by $(\ref{idasZ})$. 
\end{proof} 

 \newpage
 \section{A class of multiple zeta values (elementary version)}  \label{sectAclassofMZVs}

The class of single-valued multiple zeta values is constructed in this section in a completely `elementary' way, i.e., with no reference to motivic periods. 

 \subsection{Deligne's canonical associator}
 Consider the continuous antilinear map
 \begin{eqnarray}
 \sigma : \C\langle \langle e_0, e_1 \rangle \rangle  & \To &  \C\langle \langle e_0, e_1 \rangle \rangle \\
 \sigma (e_i) & \mapsto & -e_i  \nonumber 
 \end{eqnarray} 
which acts by complex conjugation on the coefficients of words.  Let $Z\in \R\langle \langle e_0, e_1 \rangle \rangle$ denote the Drinfeld associator $(\ref{DrinfeldAssoc})$.

 \begin{lem} There  exists a unique element  $W \in \R\langle \langle e_0, e_1 \rangle \rangle$ such that 
 \begin{equation} \label{Wdef}
   W \circ {}^\sigma \! Z  = Z\ .
 \end{equation} 
 \end{lem}
\begin{proof}  By \S\ref{sectIharaaction}, the Ihara action is transitive and faithful. The equation $(\ref{Wdef})$ can be solved
recursively using $(\ref{Frecursive})$ and the  comments which follow.
\end{proof}
  The  series  $W$ is Deligne's associator. We  show in \S\ref{sectsinglevaluedmotivicMZVs} that it is indeed an associator. 
  
 \subsection{Single-valued multiple polylogarithms}  \label{sectSVmultpoly} We briefly recall the construction given in \cite{BSVMP}. See \S\ref{SVMPrevisited} for a more conceptual derivation.  The conventions for iterated integrals will be switched relative to \cite{BSVMP}     in order
 to remain compatible with the above. The generating series of multiple polylogarithms is
$$L_{e_0,e_1}(z) = \sum_{w \in \{e_0,e_1\}^{\times}} L_w(z) w\ , $$
and is defined to be the unique solution to the K-Z equation
$$ {d\over dz}  L_{e_0, e_1}(z) =  L(z)\Big( {e_0 \over z} + {e_1 \over 1-z}\Big)$$ 
which is  equal to $ h(z) \exp( e_0 \log(z))$ near the origin,   where $h(z)$ is a holomorphic function at $0$, where it takes the value $1$.

\begin{defn} There is a unique  element $e_1' \in \opi(\R)= \R\langle \langle e_0, e_1 \rangle \rangle$  which satisfies the  fixed-point equation:
 \begin{equation} \label{olddefe1}
   Z(-e_0,-e_1') e_1' Z(-e_0, -e_1')^{-1} = Z(e_0,e_1) e_1 Z(e_0,e_1)^{-1}\ .
  \end{equation} 
 \end{defn} 
One can easily show that $(\ref{olddefe1})$ can be solved recursively in the weight, and so $e_1'$ does indeed exist and is unique \cite{BSVMP}.
 \\
 
 The generating series of single-valued multiple polylogarithms was defined by
  \begin{equation} \label{Lodef}
  \Lo(z) =   \widetilde{L}_{e_0,e_1'}(\overline{z}) L_{e_0,e_1}(z) \ , 
  \end{equation}
  where \,\, $\widetilde{}$\,\,  denotes reversal of words. Since the antipode in the Hopf algebra $\C\langle \langle e_0, e_1 \rangle \rangle$ is given by 
  $e_{i_1}\ldots e_{i_n} \mapsto (-1)^n e_{i_n}\ldots e_{i_1}$, and since $L(z)$ is group-like, we have 
  $$  L_{e_0,e_1}(z)^{-1} = \widetilde{L}_{-e_0,-e_1}(z) $$
  and therefore we can rewrite $(\ref{Lodef})$ as
$$ \Lo(z) =  (L_{-e_0,-e_1'}(\overline{z}))^{-1}  L_{e_0,e_1}(z) \ ,$$
 In \cite{BSVMP} it was shown that the coefficients $\Lo_w(z)$  of $w$ in the generating series  $\Lo(z)$ are single-valued functions of $z$, are linearly independent over $\C$, 
 and satisfy the same shuffle and differential equations (with respect to $\partial \over \partial z$) as $L_w(z)$. The last  two properties are obvious from $(\ref{Lodef})$. Their values at one are given by 
 \begin{equation} \label{Lo1} 
  \Lo(1) =  (Z(-e_0,-e_1'))^{-1}  Z(e_0,e_1)\ .
  \end{equation}

 \subsection{Values of single-valued multiple polylogarithms at 1}  The values of single-valued multiple polylogarithms at 1 are exactly the coefficients of Deligne's associator.
 \begin{lem} Equation  $(\ref{olddefe1})$ has the unique solution  
 $e_1 ' = W e_1 W^{-1}$. \end{lem} 
 \begin{proof} By the formula for the Ihara action $(\ref{Iharaact})$, we have
  \begin{equation} \label{WcircZproof}
 W \circ {}^{\sigma}\!Z(e_0,e_1) =  {}^{\sigma}\!Z(e_0,W e_1 W^{-1})  W
 \end{equation}
Let $e_1' = W e_1 W^{-1}$, and write $Z' = Z(e_0, e_1')$, and ${}^{\sigma}\!Z' = Z(-e_0, -e_1')$. We have
\begin{equation} \label{Wasmult}
{}^{\sigma}\! Z'  \overset{(\ref{WcircZproof})}{=} (W \circ {}\!^{\sigma}\!Z)  W^{-1} \overset{(\ref{Wdef})}{=} Z  W^{-1}
\end{equation}
which implies that
$ {}^{\sigma}\! Z' e_1'\, {}^{\sigma}\! (Z')^{-1} =  Z W^{-1}  (W  e_1 W^{-1})  W Z^{-1} = Z e_1 Z^{-1}$.

For the uniqueness, any solution to $(\ref{olddefe1})$ is of the form $e_1'= A e_1 A^{-1}$ for some series $A$ with leading coefficient 1, and 
$(\ref{olddefe1})$ is just $(A \circ {}^{\sigma}\!Z ) e_1 (A \circ {}^{\sigma}\!Z ) ^{-1}= Ze_1 Z^{-1}$. This readily implies that 
 $A\circ {}^{\sigma}\!Z=Z$ and so $A=W$ by $(\ref{Wdef}).$
\end{proof}
 By equation $(\ref{Lo1})$, $\Lo(1)$ is $({}^{\sigma} \!Z')^{-1}  Z$, which  is exactly $W$ by  $(\ref{Wasmult})$.

\begin{cor} \label{corLo1isW} $\Lo(1) =W.$ 
\end{cor} 

 \section{The single-valued  associator (motivic version)}

\subsection{Single-valued motivic multiple zeta values}\label{sectsinglevaluedmotivicMZVs}
Recall that $Z^{\mm} \in \HH(\Ho)$ is the motivic Drinfel'd associator $(\ref{idasZ})$, and that the action of $\sigma$ (definition  \ref{defsigma}) on the ring $\Ho$ of  coefficients 
of $\HH(\Ho)$ is denoted by ${}^{\sigma}$.

\begin{lem} There exists a unique $W^{\mm} \in \A(\Ho)$ such that
\begin{equation} \label{Wmotdef}
   W^{\mm} \circ {}^\sigma \! Z^{\mm}  = Z^{\mm}\ .
 \end{equation} 
 \end{lem}
\begin{proof}    Using a decomposition $Z^{\mm}=\Zo^{\mm} \circ \gamma^{\mm} $   $(\ref{ZisZocircgamma})$,  set $W^{\mm} = \Zo^{\mm} \circ  ({}^\sigma\! \Zo^{\mm})^{\circ -1}$,
  where the inversion and multiplication take place in the group $\A(\Ho)$. Since $\sigma$ acts trivially on the coefficients of $\gamma^{\mm}$, it is independent of the chosen decomposition: replacing $(\Zo^{\mm}, \gamma^{\mm})$ with $(\Zo^{\mm} \circ h , h^{\circ -1} \circ \gamma^{\mm})$, where ${}^{\sigma}\!h=h$,  gives rise to the same element. 
  
  Then $W^{\mm}\circ {}^\sigma\! \Zo^{\mm}=\Zo^{\mm}$,
  which implies $(\ref{Wmotdef})$.
\end{proof}

Via $\A(\Ho) = \Hom( \Ao, \Ho)$, we view $W^{\mm}$ as an algebra morphism
$$W^{\mm} : \Ao \To \Ho$$
Composing with $\Q\langle e^0, e^1 \rangle \rightarrow \Ho \rightarrow \Ao$ gives a map we denote by  $W^{\mm}:\Q\langle e^0, e^1 \rangle \rightarrow \Ho$.

\begin{defn}For every word $w \in \{e^0, e^1\}$, define  the single-valued motivic multiple zeta value to be the  image of $w$ under the map $W^{\mm}$. Denote it by
$$\zeta_{\sv}^{\mm}(w)\in \Ho\ .$$
 Let $\Ho^{\sv}\subset \Ho$ be the algebra spanned by the  $\zeta_{\sv}^{\mm}(w)$.
\end{defn} 
The fact that the map $\Q\langle e^0, e^1\rangle \rightarrow \Ho^{\sv}$ factors through $\Ao$ means  the following.
\begin{cor} The elements $\zetam_{\sv}(w)$  satisfy all the  motivic relations between motivic multiple zeta values, together with the relation
$$\zetam_{\sv}(2)=0\ .$$ 
\end{cor}
In particular, the  $\zetam_{\sv}(w)$ satisfy the usual double shuffle equations and associator relations. 
Note, however, that the map $\Ao \rightarrow \Ho^{\sv}$ is  not injective, so the elements $\zetam_{\sv}$ satisfy many more relations than their non single-valued counterparts.

We can write the element $W^{\mm}$ as a generating series:
$$W^{\mm} = \sum_w \zetam_{\sv}(w) \,w $$
Its period $\per (W^{\mm})$ is precisely $W$  defined in $(\ref{Wdef})$, which shows that $W$ is an   associator.  Therefore, by corollary \ref{corLo1isW},  
the period of $\zetam_{\sv}(w)$ is given by  the value at $1$ of the corresponding single-valued multiple polylogarithm:
\begin{equation}
\per ( \zetam_{\sv}(w)) = \Lo_w (1) \ . 
\end{equation}

\subsection{Structure of $\Ho^{\sv}$}  \label{sectStructure}
Let $\Ao^{\sv} \subset \Ao$ denote the image of $\Ho^{\sv}$ under the map $\pim$.
\begin{lem}  \label{lemHsvisAsv}
The map $\pim : \Ho^{\sv} \rightarrow \Ao^{\sv}$ is an isomorphism, and  $\Ho^{\sv} \cong \Ao^{\sv}$ is the image of $\Ao$ under the homomorphism
\begin{equation} \label{idsigmaprojection}
\sv_{\!\Ao}:= \id  \circ \sigma^{\circ -1} : \Ao \To \Ao
\end{equation}
where the  multiplication $\circ$ and inverse take place in the group $\A(\Ao)$.
\end{lem}
\begin{proof}  This follows from lemma \ref{imageofsvm} since $\Ho^{\sv} \subset \Pe^{\sv}$.
\end{proof} 

 Denote the  Lie coalgebra of indecomposable elements of $\Ao$ by 
$$\Lo=  {\Ao_{>0}  \over  \Ao_{>0}\Ao_{>0}}  \ .$$
Since $\Ao$ is a commutative, graded Hopf algebra, it follows from standard facts that $\Ao$ is isomorphic to the polynomial
algebra generated by the elements of $\Lo$.

\begin{prop} \label{propmain} The algebra $\Ho^{\sv}$ of single-valued motivic multiple zeta values is isomorphic to the polynomial algebra  
generated by elements $\Lo^{\mathrm{odd}}$ of $\Lo$ of odd weight.
\end{prop} 
\begin{proof}
By lemma  $\ref{lemHsvisAsv}$, $\Ho^{\sv} \cong  \sv_{\!\Ao}(\Ao)$, where $  \sv_{\!\Ao}= \id  \circ \sigma^{\circ -1}$. Since $\sv_{\!\Ao}$ is a homomorphism, it defines a map
$\sv_{\!\Lo} : \Lo \rightarrow \Lo$.
Writing $\sv_{\!\Ao} = \mu ( \id \otimes \sigma^{\circ -1}) \Delta$, where $\Delta : \Ao \rightarrow \Ao\otimes \Ao$ is the coproduct, and $\mu$ is the multiplication on $\Ao$, we see that 
$$\sv_{\!\Ao} \equiv \id - \sigma \quad \hbox{Ê(modulo  products)}$$
and therefore 
$\sv_{\!\Lo} = 2 \, \pi^{\mathrm{odd}}$
where $\pi^{\mathrm{odd}} : \Lo \rightarrow \Lo^{\mathrm{odd}}$ is the projection onto the part of odd weight. It follows that  $\Ao^{\sv}$ is multiplicatively generated by $\Lo^{\mathrm{odd}}$. 
\end{proof}


\subsection{Single-valued multiple polylogarithms revisited} \label{SVMPrevisited}
 We can re-derive the construction of the single-valued multiple polylogarithms of \cite{BSVMP} and \S\ref{sectSVmultpoly}
as follows. 

Consider $X= \Pro^1 \backslash \{0,1,\infty\}$, with the base-points $\{\tone_0, -\tone_1, z\}$ for some $z\in X(\C)$.
Its de Rham fundamental groupoid consists of a copy of $\Q\langle e_0, e_1 \rangle$  for each pair of these base-points. We  shall only consider the copies $\opo, \opi, \opz$,
 corresponding to the canonical de Rham paths between the base-points indicated by their subscripts.

Let $\Aut$ denote the group of automorphisms of $\pi_1(X,\{\tone_0, -\tone_1, z\})$  which preserves the copy of  $e_0$  in $\opo$ and $e_1$  in $\opi$. Modifying \cite{DG}, Proposition 5.9 accordingly,   the  action of $\Aut$ on the elements  $\ooi , \ooz$  defines  an injective map
\begin{eqnarray}
  \Aut  & \hookrightarrow & \opi \times \opz \\
  a &\mapsto  & (a_1, a_z)   \nonumber
  \end{eqnarray}
  where for any $a\in \Aut$, we write $a_z= a(\ooz)$ and $a_1 = a(\ooi)$. We leave to the reader the verification that this is an isomorphism. 
  Since $\opz$ is a left $\opo$-torsor, we immediately deduce a formula for the generalized  Ihara action
\begin{eqnarray} \label{zIhara}
\Aut \times \opz  & \To & \opz \nonumber \\
a \circ b & = &  \langle a_1 \rangle_0 (b)\,  . \, a_z
\end{eqnarray} 
where $\langle a  \rangle_0$ denotes the  action of $a \in \opi$ on $\opo$ (\cite{DG}, (5.9.4)). Concretely, this gives
\begin{eqnarray} \label{Iharaz}
\qquad \quad  (\opi \times  \opz ) \times (\opi \times  \opz )  & \To& (\opi \times  \opz )  \\
(F_1, F_z) \,\, \circ \,\,(G_1,G_z) & = & (G_1(e_0, F_1 e_1 F_1^{-1}) F_1 ,   G_z(e_0, F_1 e_1 F_1^{-1}) F_z) \nonumber 
\end{eqnarray}
The action  of $\opi \times  \opz $ on $\opi$  factors through the usual Ihara action of $\opi$ on $\opi$.

Let us fix a path $\mathrm{ch}_z$ from $\tone_0$ to $z$ in $X(\C)$.  Its de Rham image in $\opz(\C)$ is exactly (some branch of)  the generating series of multiple polylogarithms (\S\ref{sectSVmultpoly})
$$\mathrm{ch}^{dR}_z = L(z) \, \, \in\,  \opz(\C)$$  
By the general single-valued principle, we seek an element $W =(W_1,W_z)$ in  the group $\Aut(\C) \cong \opi(\C)  \times \opz(\C) $ such that 
$$ 
 W \circ ({}^{\sigma}\! Z,  {}^{\sigma}\! L(z) ) =  (Z, L(z)) \ .
 $$
It has a solution since $\opi \times \opz$ is a torsor over $\Aut\cong  \opi \times \opz$.
By the formula for the action $(\ref{Iharaz})$, this is equivalent to the pair of equations
\begin{eqnarray} \label{Wzequation} 
  {}^{\sigma}\! L_{e_0, W_1 e_1 W_1^{-1} }(z)W_z & =  & L(z)\ .  \\
   {}^{\sigma}\! Z(e_0, W_1 e_1 W_1^{-1} ) W_1 & =  & Z\ , \nonumber 
  \end{eqnarray}
and so $W_1 \circ {}^{\sigma}\! Z= Z$, and $W_1$ is equal to the element $W$ defined in $(\ref{Wdef})$. As in \S\ref{sectSVmultpoly}, write  $e_1'=W e_1 W^{-1}$. Therefore by $(\ref{Wzequation})$ we deduce the following formula for $W_z$,
$$W_z =    L_{-e_0, -e_1'}^{-1} (\overline{z})L(z)  $$
It is independent of the choice of path $\mathrm{ch}_z$, and is therefore single-valued. This gives another derivation of the construction in \cite{BSVMP}.

 \section{Generators for $\Ho^{\sv}$ and examples} \label{sectGeneratorsandEx}
 Up to this point we have used no deep results about the category of mixed Tate motives, nor about the structure of motivic multiple zeta values. 

\subsection{Periods of mixed Tate motives}
  In \cite{BMTZ}, it was shown  that 
  \begin{equation} \label{Aoisom} \Ao \cong  \Or(U_{dR})= \Pe^{\am}\ .
  \end{equation}
   The following  proposition, due to Deligne \cite{DLetter},  is a more precise statement about periods of mixed Tate motives than the one stated in \cite{BMTZ}.

\begin{prop} \cite{DLetter} Let $M\in \MT(\Z)$ be a mixed Tate motive over $\Z$ with non-negative weights, i.e., $W_{-1} M =0$. Let 
$\eta \in (M_{dR})_n$ and $X \in M_B^{\vee}$.  

i). If $c(X)=X$ then the motivic period $[M, \eta, X]^{\mm}$ is a rational linear combination of motivic multiple zeta values of weight $n$.

ii). If $c(X)=-X$ then the motivic period $[M, \eta, X]^{\mm}$ is a rational linear combination of motivic multiple zeta values of weight $n-1$, multiplied by $\LM$.
\end{prop}
\begin{proof}  By $(\ref{Aoisom})$, and  \S\ref{sectdecompzeta(2)}, 
  $\Ho \cong \Pe^{\am} \otimes \Q[(\LM)^2] \cong \Pemp_{\R}.$
The  result then follows immediately  from the definitions of $\Pemp_{\R}$ and $\Pemp_{i\R}$, and corollary \ref{corRealFrob}. 
\end{proof}
\noindent 
The methods of \cite{DLetter} give an equivalent but slightly different proof of corollary $\ref{corRealFrob}$.

 \subsection{A model for $\Ho^{\sv}$} 
Applying  a choice of trivialization $(\ref{Pemdecomp})$ and a choice of generators for $\Or(U_{dR})$  to  $(\ref{Aoisom})$ gives a non-canonical isomorphism  \cite{BMTZ}, \S2.5:
  \begin{equation}\label{HisomU}
  \Ho \cong  \U\otimes \Q[f_2]
  \end{equation}
  such that the natural map $\pim: \Ho \rightarrow \Ao$ induces an isomorphism $U \cong \Ao$, 
  and where
  $$\U = \Q\langle f_3, f_5, f_7, \ldots \rangle $$
  is the graded Hopf algebra cogenerated by one generator $f_{2n+1}$ in every odd degree $2n+1\geq 3$, equipped with the shuffle product and 
  the deconcatenation coproduct
  $$\Delta_{dec}( f_{i_1} \ldots f_{i_n}) = \sum_{k=0}^n f_{i_1} \ldots f_{i_k} \otimes f_{i_{k+1}} \ldots f_{i_n}$$
   The element $f_2$ corresponds to $(\LM)^2$ and satisfies $\Delta(f_2) =1 \otimes f_2$.
   
  The map $\sv: \Ao \rightarrow \Ao$ defines a homomorphism  
  \begin{eqnarray} \label{svonU}
  \sv : \U & \To & \U  \\
  w & \mapsto & \sum_{uv=w} u \sha \widetilde{v}\nonumber
  \end{eqnarray} 
  where $\,\,\widetilde{ }\,\,$ denotes reversal of words.  To see this, note that $\sigma: \U \rightarrow \U $ is the map $f_{2n+1} \mapsto - f_{2n+1}$, and the antipode $S$ on $\U$ is 
  given by $w \mapsto \sigma(\widetilde{w})$.
  By formula $(\ref{idsigmaprojection})$, we have $\sv = \mu ( \id \otimes  {}^{\sigma}\! S) \Delta_{dec}$, where $\mu$ is multiplication, which immediately gives  $(\ref{svonU})$.
  
  Since  $\pim : \Ho^{\sv} \cong \Ao^{\sv}$ by lemma \ref{lemHsvisAsv}, we conclude that
  $$\Ho^{\sv} {\cong} \Ao^{\sv} \cong \U^{\sv}$$
  where $\U^{\sv}$ is the image of the map $(\ref{svonU})$. By way of example,
  $$ \sv(f_{a})= 2 f_{a} \quad  \ , \quad  \sv(f_{a}f_{b})=  2 (f_{a} f_{b}+f_{b} f_{a}) $$
  $$\sv(f_{a} f_{b} f_{c})  = 2( f_{a} f_{b} f_{c} + f_{a} f_{c} f_{b} + f_{c}  f_{a} f_{b} +f_{c}  f_{b} f_{a} )$$
  where $a,b,c$ are odd integers $\geq 3$. In general, we have  the formula $(\ref{deltasv})$, which gives
  $$\sv(f_a w f_b) =  f_a \,\sv(wf_b) +  f_b\, \sv (f_a w)$$
  for any word $w \in \{f_{2n+1}\}$, and $a,b$ odd integers $\geq 3$.  This follows immediately from $(\ref{svonU}),$ since, via the recursive definition of $\sha$, we have
  $$ f_a u \sha f_b \widetilde{v} = f_a ( u \sha f_b \widetilde{v}) + f_b ( f_a u \sha \widetilde{v})\ .$$
  
  \subsection{Hoffman-type generators for $\Ho^{\sv}$} 
    Let $V$ be a finite ordered set. A Lyndon word  in the elements of $V$ is a word which is  smaller in the lexicographic ordering than its strict right factors:
    if $w= uv$, then $w<v$ whenever $u, v$ are non-empty.
  
  In (\cite{BMTZ}, \S8), the following theorem was proved.
  
  \begin{thm} \label{thmHoff} The ring of motivic multiple zeta values $\Ho$ is generated by the Hoffman-Lyndon elements
  $\zetam(w)$  where $w$ is a Lyndon word in the alphabet $\{2,3\}$, where $3<2$.
  \end{thm} 
  It immediately follows from proposition \ref{propmain} that 
  \begin{cor}The ring of  single-valued motivic multiple zeta values $\Ho^{\sv}$ is generated by the Hoffman-Lyndon elements
  $$\zetam_{\sv}(w)$$  where $w$ is a Lyndon word of odd weight  in the alphabet $\{2,3\}$, where $3<2$.
  \end{cor}
  A Hoffman-Lyndon word of odd weight necessarily has an odd number of $3$'s.
  It follows from theorem \ref{thmHoff} that the Poincar\'e series of $\Ho$ is given by 
  $$
  \sum_{n\geq 0 } \dim \Ho_n \, t^n  = {1 \over 1-t^2-t^3}\ .
  $$
  The dimensions $\ell_n = \dim \Lo_n$ of the Lie coalgebra $\Lo$ are determined by 
  $$ \prod_{n\geq 1} (1-t^n)^{-\ell_n} = { 1 \over 1 -t^2-t^3}\ .$$
  The numbers $\ell_n$ can be interpreted  either as the number of Lyndon words of weight $n$ in $\{2,3\}$, where $3<2$, or as the number of Lyndon words of weight $n$
  in the alphabet $\{f_3<f_5< f_7,\ldots, \}$, via the isomorphism $(\ref{HisomU})$.
   By proposition \ref{propmain},
  \begin{cor}
  The Poincar\'e series of $\Ho^{\sv}$ is given by 
  $$ \sum_{n\geq 0 } \dim \Ho^{\sv}_n \, t^n  =  \prod_{n\, \mathrm{odd} \geq 1} (1-t^n)^{-\ell_n} \ .$$
  \end{cor}
 \subsection{Examples} \label{sectExamples}
 For the convenience of the reader, we list the dimensions of the space of motivic multiple zeta values $\Ho$ and its version modulo products $\Lo$:

 {\small
 $$
\begin{array}{c|cccccccccccccccccccc}
N & 1 & 2   & 3  & 4 & 5   & 6 & 7 & 8 & 9 & 10 & 11 & 12 & 13   & 14 & 15 & 16 & 17 & 18 & 19 & 20 \\ \hline
\dim \Lo_N &    0 &  1 & 1 & 0 & 1&  0&  1&  1&  1&  1&  2&  2&  3&  3&  4&  5&  7&  8&  11&  13   \\
\dim \Ho_N  &   1 &1 & 1 &4 &2  & 2 & 3& 4 & 5 & 7 & 9 & 12& 16 &21 & 28 & 37 & 49 & 65 &86 &114   \\
\end{array}
$$}

 \noindent Next,  their single-valued
 versions $\Ho^{\sv}$ and $\Lo^{\sv}$:

 {\small
 $$
\begin{array}{c|cccccccccccccccccccc}
N & 1 & 2   & 3  & 4 & 5   & 6 & 7 & 8 & 9 & 10 & 11 & 12 & 13   & 14 & 15 & 16 & 17 & 18 & 19 & 20 \\ \hline
\dim \Lo^{\sv}_N &    0 &  0 & 1 & 0 & 1&  0&  1&  0 &  1&  0&  2& 0&  3&  0&  4&  0&  7&  0&  11&  0   \\
\dim \Ho^{\sv}_N  &  1 & 0 & 1 & 0  &  1 & 1 & 1 & 1 &2 & 2 & 3 & 3 & 5 & 5 & 8 &8 & 13 & 14 & 21 & 23  \\
\end{array}
$$}

Note that  $\dim_{\Q} \Ho^{\sv}_N$ happens to equal $\dim_{\Q} \Lo_{N+2}$ for $1\leq N \leq 12$, which adds to the large supply of  evidence for exercising caution when identifying integer sequences!

Below we list   algebra generators for $\Ho^{\sv}_N$ for $1\leq N\leq 14$. They were calculated by Oliver Schnetz using \cite{SG}, which gives a very efficient way to compute $(\ref{Wdef})$  \cite{SI}. 
$$\begin{array}{|c|c|c|c|c|c|c|c|}
\hline
  N &  3   & 5 &   7 &   9 &   11 &  13 \\
  \hline
\hbox{Generators}  &  \zetam_{\sv}(3) & \zetam_{\sv}(5) &    \zetam_{\sv}(7) &  \zetam_{\sv}(9) &  \zetam_{\sv}(11) &  \zetam_{\sv}(13)  \\
\hbox{ of } & & &  &   &  \zetam_{\sv}(3,5,3) &  \zetam_{\sv}(5,3,5)  \\
\Ho_{\sv} & & &  &   &  &  \zetam_{\sv}(3,7,3)  \\
  \hline
\end{array}
$$
Here, $\zetam_{\sv}(2n+1) = 2 \,\zetam(2n+1)$ for all $n\geq 1$, and 
\vspace{0.05in}
{
 \begin{eqnarray} \label{zmsvexamples} 
\zetam_{\sv} (3,5, 3) & =&  2 \zetam(3,5,3) - 2 \zetam(3) \zetam(3,5) -10   \zetam(3)^2 \zetam(5) \\
\zetam_{\sv} (5,3, 5) & =&  2 \zetam(5,3,5) - 22 \zetam(5) \zetam(3,5)- 120 \zetam(5)^2\zetam(3)  \nonumber \\
& &-10  \zetam(5) \zetam(8)  \nonumber   \\ 
\zetam_{\sv} (3,7,3) & =&   2 \zetam(3,7,3) -2 \zetam(3) \zetam(3,7) - 28 \zetam(3)^2\zetam(7)    \nonumber \\
&  &   -24 \zetam(5) \zetam(3,5) - 144 \zetam(5)^2 \zetam(3)     - 12 \zetam(5) \zetam(8) \ .  \nonumber 
\end{eqnarray}
}
\vspace{-0.1in}

  \begin{rem} The  generating series of  unipotent de Rham multiple zetas  in depth $r$ is
 \begin{eqnarray}
 Z_r(x_1,\ldots, x_r)= \sum_{n_1,\ldots, n_r \geq 1} \zetaa(n_1,\ldots, n_r) x_1^{n_1-1}\ldots x_r^{n_r-1}\ .
 \end{eqnarray} 
 Let $Z^{\sv}_r$ denote the corresponding  single-valued version.  Then using the methods of \cite{BrDepth} we 
 can  verify that
 \begin{eqnarray}
 Z_1^{\sv}  & =  & Z_1 - {}^{\sigma}\! Z_1 \nonumber \\
 Z_2^{\sv}& \equiv &  Z_2 - {}^{\sigma}\! Z_2-  2 Z_1\circb {}^{\sigma}\!  Z_1    \nonumber \\
  Z_3^{\sv}   &\equiv  & Z_3 - {}^{\sigma}\! Z_3-  2  Z_1  \circb {}^{\sigma}\! Z_2 -   2 Z_1 \circb (Z_1 \circb {}^{\sigma}\! Z_1)
  \nonumber 
 \end{eqnarray}
 where the equivalence sign means modulo $\zetam(2)$ and modulo terms of lower depth, and 
 where $\circb$ is the linearized Ihara operator defined in \cite{BrDepth}, \S6.  For example:
\begin{eqnarray}
f(x_1) \circb g(x_1) & =  &f(x_{1})g(x_{2})+f(x_{2}-x_{1})\big( g(x_1)-g(x_{2}) \big) \nonumber \\
f(x_1) \circb g(x_1,x_2)  &=  & f(x_{1}) g(x_{2}, x_{3})+f(x_{2}-x_{1}) (g(x_{1}, x_{3})-g(x_{2}, x_{3}))\nonumber \\
 & &+ \qquad f(x_{3}-x_{2}) (g(x_{1}, x_{2})-g(x_{1}, x_{3}))  \nonumber 
\end{eqnarray}
 In  particular, this confirms the formulae  $(\ref{zmsvexamples})$ in odd weights, modulo $\zetam(2)$.
 \end{rem}

\bibliographystyle{plain}
\bibliography{main}

\end{document}